\documentclass[11pt,pdflatex]{amsart}
\usepackage[usenames]{color}
\usepackage{color}
\usepackage{amssymb}
\usepackage{graphicx, epsfig}
\usepackage{latexsym, amsfonts, amscd, amsmath}
\usepackage{mathrsfs}
\makeindex 
\input xy
\xyoption{all}

\definecolor{orange}{rgb}{1,0.5,0}
\definecolor{Indigo}{rgb}{0.2,0.1,0.7}
\definecolor{Violet}{rgb}{0.5,0.1,0.7}

\newtheorem{thm}{Theorem}[section]
\newtheorem{prop}[thm]{Proposition}
\newtheorem{lem}[thm]{Lemma}
\newtheorem{cor}[thm]{Corollary}

\theoremstyle{definition}
\newtheorem{dfn}[thm]{Definition}

\theoremstyle{remark}
\newtheorem{rmk}[thm]{Remark}

\newenvironment{mat}{\left(\begin{smallmatrix}
\end{smallmatrix}\right)}{enddef}

\numberwithin{equation}{section}
\numberwithin{figure}{subsection} \numberwithin{table}{subsection}

\newcommand{\codim}{{\operatorname{codim }}}

\newcommand{\Fr}{{\operatorname{Fr }}}

\newcommand{\Ker}{{\operatorname{Ker}}}

\newcommand{\val}{{\operatorname{val}}}
\newcommand{\Ver}{{\operatorname{Ver }}}

\newcommand{\GL}{{\operatorname{GL}}}

\newcommand{\Lie}{{\operatorname{Lie}}}

\newcommand{\gera}{{\frak{a}}}

\newcommand{\gerd}{{\frak{d}}}

\newcommand{\germ}{{\frak{m}}}

\newcommand{\gerp}{{\frak{p}}}
\newcommand{\gerq}{{\frak{q}}}

\newcommand{\gert}{{\frak{t}}}

\newcommand{\gerC}{{\frak{C}}}
\newcommand{\gerD}{{\frak{D}}}

\newcommand{\gerM}{{\frak{M}}}

\newcommand{\gerS}{{\frak{S}}}

\newcommand{\gerX}{{\frak{X}}}
\newcommand{\gerY}{{\frak{Y}}}
\newcommand{\gerZ}{{\frak{Z}}}

\newcommand{\uA}{{\underline{A}}}
\newcommand{\uB}{{\underline{B}}}

\newcommand{\calA}{{\mathcal{A}}}

\newcommand{\calC}{{\mathcal{C}}}

\newcommand{\calF}{{\mathcal{F}}}

\newcommand{\calI}{{\mathcal{I}}}

\newcommand{\calM}{{\mathcal{M}}}

\newcommand{\calO}{{\mathcal{O}}}

\newcommand{\calR}{{\mathcal{R}}}

\newcommand{\calU}{{\mathcal{U}}}
\newcommand{\calV}{{\mathcal{V}}}
\newcommand{\calW}{{\mathcal{W}}}

\def\BB{\mathbb{B}}
\def\CC{\mathbb{C}}
\def\DD{\mathbb{D}}

\def\GG{\mathbb{G}}
\def\HH{\mathbb{H}}
\def\II{\mathbb{I}}

\def\QQ{\mathbb{Q}}
\def\RR{\mathbb{R}}
\def\SS{\mathbb{S}}

\def\ZZ{\mathbb{Z}}

\newcommand{\scrP}{{\mathscr{P}}}

\newcommand{\arr}{{\; \rightarrow \;}}
\newcommand{\Arr}{{\; \longrightarrow \;}}

\newcommand{\ol}{{\mathcal{O}_L}}

\newcommand{\Spf}{\operatorname{Spf }}

\newcommand{\Xrig}{{\gerX_\text{\rm rig}}}

\newcommand{\Yrig}{{\gerY_\text{\rm rig}}}

\newcommand{\spe}{{\operatorname{sp }}}

\newcommand{\Xbar}{\overline{X}}
\newcommand{\Ybar}{\overline{Y}}

\newcommand{\pe}{(\varphi,\eta)}
\newcommand{\peprime}{(\varphi',\eta')}
\newcommand{\Wpe}{W_{\varphi,\eta}}
\newcommand{\Zpe}{Z_{\varphi,\eta}}

\newcommand{\Qbar}{\overline{Q}}
\newcommand{\Pbar}{\overline{P}}

\newcommand{\Qpbar}{\overline{\QQ}_p}

\newcommand{\Sib}{{\rm{Sib}}}

\newcommand{\tYrig}{\tilde{\gerY}_{\rm{rig}}}
\newcommand{\tXrig}{\tilde{\gerX}_{\rm{rig}}}
\newcommand{\tYrigtau}{\tilde{\gerY}^{|\tau|\leq 1}_{\rm{rig}}}
\newcommand{\tXrigtau}{\tilde{\gerX}^{|\tau|\leq 1}_{\rm{rig}}}
\newcommand{\Yrigtau}{\gerY^{|\tau|\leq 1}_{\rm{rig}}}

\newcommand{\tYrigp}{\tilde{\gerY}_{\rm{rig}}^\gerp}

\begin{document}
\marginparwidth 50pt

\openup 1pt

\title[Modularity lifting in parallel weight one]{Modularity lifting in parallel weight one}
\date{September 2011}
\author{Payman L Kassaei}
\address{Department of Mathematics, King's College London,
Strand, London WC2R 2LS, United Kingdom.}
\email{payman.kassaei@kcl.ac.uk}
\subjclass{Primary [11F80, 11F33, 11F41]; Secondary [11G18, 14G35, 14G22].}

\begin{abstract}
We prove an analogue of the main result of Buzzard and Taylor \cite{BuzzardTaylor} for totally real fields in which $p$ is unramified. This can be used to prove certain cases of the strong Artin conjecture over totally real fields.
\end{abstract}

\maketitle

\section{Introduction}



In \cite{BuzzardTaylor}, Buzzard and Taylor proved a modularity lifting theorem for
certain Galois representations in weight 1, combining methods of
Wiles and Taylor (\`a la Diamond \cite{Diamond2}) with analytic continuation of
overconvergent modular forms. This work was part of a program laid
out by Taylor in \cite{Taylor} to prove the strong Artin conjecture in many
icosahedral instances.

The work of Buzzard and Taylor highlighted the importance of the
geometric interpretation of  overconvergent modular forms, and
proved useful the idea of analytic continuation. Later, in
\cite{Kassaeicoleman}, we presented another application of  analytic continuation by giving a
geometric proof of Coleman's classicality theorem. This method  has been used
to prove classicality criteria for overconvergent automorphic forms
in several contexts. One can mention \cite{Kassaeiclassicality},
\cite{Sasaki}, \cite{Ramsey}, \cite{Pilloni}, and, quite recently,
the classicality results in the unramified Hilbert case
\cite{PilloniStroh}, \cite{Tian}.

In this article, we use ideas from our work on analytic continuation
of overconvergent Hilbert modular forms to prove an analogue of the main result of
\cite{BuzzardTaylor} in the unramified  Hilbert case. The case where $p$ is split in the totally real field has been proven by Sasaki \cite{Sasaki2}. To understand
the dynamic of the $U_p$ Hecke correspondence, we use our results
(joint with E. Goren) on the geometry of the Iwahori-level Hilbert
modular variety studied in \cite{GorenKassaei}, as well as
calculations using Breuil-Kisin modules. We should mention that
Breuil-Kisin module calculations in this context have also been used
by Y. Tian in \cite{Tian}, and first appeared in the work of S.
Hattori on canonical subgroups \cite{Hattori}.

Since, unlike the classical case, analytic continuation of  overconvergent Hilbert modular forms in the
locus of non-ordinary reduction of the Hilbert modular variety is dependent upon slope
conditions which are impossible to hold in this case, the processes 
of analytic continuation and gluing are subtler. For example, in our argument, it is crucial to prove that overconvergent  Hilbert modular forms of finite slope can be analytically extended to a certain region in the ``middle parts" of the p-adic rigid analytic Iwahori-level  Hilbert modular variety $\Yrig$ which is saturated with respect to the map $\pi:\Yrig \arr \Xrig$ that forgets the level structure at $p$. This region can be thought of as an analogue of the circle of radius $p/(p+1)$ in the case of modular curves. 

Let us briefly outline the strategy of the proof. For simplicity, assume now that $p$ is inert in the totally real field in question. Let $\calV \subset \Yrig$ be the region defined in \cite[\S 5.3.]{GorenKassaei}; it can be characterized as the region in $\Yrig$ classifying pairs consisting of a Hilbert-Blumenthal abelian variety and its canonical subgroup. In \cite{GorenKassaei}, we have proved that every finite slope overconvergent Hilbert modular form
can be analytically extended to $\calV$. Let $w: \Yrig \arr \Yrig$ be the Atkin-Lehner involution. Buzzard-Taylor's idea of gluing forms on $\calV$ and $w^{-1}(\calV)$ to create a Hilbert modular form on $\Yrig$ fails in this case, as $\calV \cup w^{-1}(\calV)$ is only a very small region in $\Yrig$: it misses tubes in $\Yrig$ of Zariski dense open  subsets of  all the `vertical' irreducible components of $\Ybar$, the reduction mod $p$ of a formal model of $\Yrig$. 

Let $\Ybar$ (respectively, $\Xbar$) denote  the reduction mod $p$ of a suitable formal model of $\Yrig$ (respectively, $\Xrig$). At the heart of our argument lies the  proof that every overconvergent Hilbert modular form of finite slope can be analytically extended from $\calV$ to a region $\calR \subset \Yrig$ with the property that the region $\calR \cup w^{-1}(\calV)\subset \Yrig$ is ``large" in the sense that  its complement collapses to  a closed subset of codimension two in $\Xbar$,  under the forgetful map $\pi: \Ybar \arr \Xbar$. On the other hand,  we show that the form obtained by gluing two forms on $\calR$ and $w^{-1}(\calV)$ descends under the morphism $\pi:\Yrig \arr \Xrig$ to a form, which, by the above fact, is defined on $\Xrig$ away from the tube of a codimension-two subset of $\Xbar$. We then invoke a rigid analytic Koecher principle to show that such a form automatically extends to the entire $\Xrig$, providing us with a classical Hilbert modular form of level prime to $p$, and weight $(1,1,\cdots,1)$.  The geometry of the intersection and interposition of irreducible components of the strata in $\Ybar$, studied in \cite{GorenKassaei}, is used crucially, especially in the gluing of forms on $\calR$ and $w^{-1}(\calV)$.

We now state our main result. Let $p>2$ be a
prime number, and $L$ a totally real field in which $p$ is unramified. For any prime ideal $\gerp|p$, let
$D_\gerp$ denote a decomposition group at $\gerp$.  We prove:

\begin{thm}{\label{Theorem: main-intro}} Let $\rho: Gal(\overline{\QQ}/L) \arr \GL_2(\calO)$ be a continuous representation, where $\calO$ is the ring of integers in a finite extension of $\QQ _p$, and $\germ$ its maximal ideal. Assume

\begin{itemize}

\item $\rho$ is unramified outside a finite set of primes,

\item For every prime $\gerp | p$, we have $\rho_{|_{D_\gerp}} \cong \alpha_\gerp \oplus \beta_\gerp$, where $\alpha_\gerp,\beta_\gerp: D_\gerp \arr \calO^\times$ are unramified characters, distinct modulo $\germ$,

\item $\overline{\rho}:=(\rho\ {\rm mod}\ \germ)$ is ordinarily modular, i.e., there exists a classical Hilbert modular form $h$  of parallel weight $2$, such that ${\rho}\equiv \rho_h\ (\mathrm{mod}\  \germ)$, and $\rho_{h}$ is potentially ordinary and potentially Barsotti-Tate at every prime of $L$ dividing $p$ (See Remark \ref{Remark: ordinary lift}),

\item $\overline{\rho}$ is absolutely irreducible when restricted to $Gal(\overline{\QQ}/L(\zeta_p))$.

\end{itemize}

Then, $\rho$ is isomorphic to $\rho_f$,  the Galois representation
associated to a Hilbert modular eigenform $f$ of weight
$(1,1,\cdots,1)$ and level $\Gamma_{00} (N)$, for some integer $N$
prime to $p$.

\end{thm}

\begin{rmk} \label{Remark: ordinary lift}
When $p \geq 7$, the third assumption in the statement of Theorem \ref{Theorem: main-intro} 
is automatic as long as $[L(\zeta_p):L]>4$, by \cite[Theorem 6.1.10]{BLGG}. In applications to the Artin conjecture ($p=5$) this assumption can be arranged to be satisfied.
\end{rmk}

 In \cite{KST}, we generalize Theorem  \ref{Theorem: main-intro} to allow some ramification for the characters $\alpha_\gerp,\beta_\gerp$. In  \cite{Sasaki2,Sasaki3}, Sasaki employs Taylor's  method \cite{Taylor}  to prove the mod-$5$ modularity of certain Galois representations in the case of totally real fields in which $5$ is split. His method applies equally well to the case where $5$ is unramified in the totally real field, as explained in \cite[\S 4]{KST}.  Using these results, Theorem \ref{Theorem: main-intro}  immediately implies the following cases of the Strong Artin Conjecture over totally real fields.

\begin{thm}\label{Theorem: Artin-intro} Let $L$ be a totally real field in which $5$ is unramified.
Let
\[
\rho: \mathrm{Gal}(\overline{\mathbb{Q}}/L)\rightarrow GL_2(\mathbb{C})
\]
be a totally odd and continuous representation satisfying the following conditions:
\begin{itemize}
\item $\rho$ has the projective image $A_5$.
\item For every place $\mathfrak{p}$ of $L$ above $5$,   the projective image of the decomposition group $D_\mathfrak{p}$ at $\mathfrak{p}$ has order 2. Furthermore,   the quadratic extension of $L_\mathfrak{p}$ fixed by the kernel of $D_\mathfrak{p}\hookrightarrow G_L\stackrel{\mathrm{proj}\rho}{\rightarrow} A_5$ is \emph{not} $L_\mathfrak{p}(\sqrt{5})$.
\end{itemize}
Then, there exists a holomorphic Hilbert cuspidal eigenform $f$ of weight $(1,1,...,1)$ such that $\rho$ arises from $f$ in the sense of Rogawski-Tunnell, and the Artin $L$-function
$L(\rho, s)$ is entire.
\end{thm}

 In \cite{KST}, we use our generalization of   \ref{Theorem: main-intro} to prove a more general version of Theorem \ref{Theorem: Artin-intro}.

\

{\bf{Acknowledgments.}} We would like to thank B.
Conrad, F. Diamond, S. Sasaki, and R. Taylor for helpful conversations.  We are grateful to K. Buzzard, T. Gee, and C. Johannson, as well as E. Goren for reading a first draft of this article and providing valuable feedback. We would also like to thank the
Mathematics Department at Stanford University where the author spent some time in Winter 2011, when part of this work was completed. This work was supported by EPSRC Grant EP/H019537/1.
After this work was completed, and while under review in a journal, we were sent a preprint by V. Pilloni where results similar to those in this paper are presented.

\section{Geometry of Hilbert Modular Varieties}\label{section: stratifications}
In this section, we recall some results on the mod $p$ and $p$-adic  geometry of Hilbert modular varieties studied in \cite{GorenKassaei} and \cite{GorenOort}.

\subsection{Stratifications in characteritic $p$} Let $p$ be a prime number, and  $N\geq 4$ an integer prime to $p$. Let $L/\mathbb{Q}$ be a totally real field of degree $g>1$ in which $p$ is unramified, $\mathcal{O}_L$ its ring of integers, and $\gerd_L$ its different ideal.    For a  prime ideal $\gerp$  of $\ol$ dividing $p$, let $\kappa_\gerp = \ol/\gerp$, a finite field of order $p^{f_\gerp}$. Let $\kappa$ be a finite field containing an isomorphic copy of all $\kappa_\gerp$ and generated by them. We identify $\kappa_\gerp$ with a subfield of $\kappa$ once
 and for all. Let $\mathbb{Q}_\kappa$ be the fraction field of $W(\kappa)$.
 We fix embeddings $\mathbb{Q}_\kappa \subset \QQ_p^{\rm ur}
 \subset \Qpbar$.

Let $\mathbb{B}={\rm Emb}(L,\mathbb{Q}_\kappa)=\textstyle\coprod_{\mathfrak{p}|p}
\mathbb{B}_\mathfrak{p}$, where $\mathbb{B}_\mathfrak{p}= \{\beta\in\mathbb{B}\colon
\beta^{-1}(pW(\kappa)) = \mathfrak{p}\}$, for every prime ideal $\mathfrak{p}$ dividing $p$. Let $\sigma$ denote the
Frobenius automorphism of $\mathbb{Q}_\kappa$, lifting $x \mapsto
x^p$ modulo $p$. It acts on $\mathbb{B}$ via $\beta \mapsto \sigma
\circ \beta$, and transitively on each $\mathbb{B}_\mathfrak{p}$.
For $S \subseteq \mathbb{B}$, we let $S^c=\BB-S$, and
$\ell(S)=\{\sigma^{-1}\circ\beta\colon \beta \in S\}$, the left shift of $S$. Similarly, define the right shift of $S$, denoted $r(S)$. The decomposition \[\mathcal{O}_L
\otimes_\mathbb{Z} W(\kappa)=\bigoplus_{\beta \in \mathbb{B}}
W(\kappa)_\beta,\] where $W(\kappa)_\beta$ is $W(\kappa)$ with the
$\mathcal{O}_L$-action given by $\beta$, induces a decomposition,
\[M=\bigoplus_{\beta\in \mathbb{B}} M_\beta,\] on any $\mathcal{O}_L
\otimes_\mathbb{Z} W(\kappa)$-module $M$.

\

 Let $X/W(\kappa)$ be the Hilbert modular scheme representing the functor which associates to a $W(\kappa)$-scheme $S$, the set of all  polarized abelian schemes with real multiplication and $\Gamma_{00}(N)$-structure $\underline{A}/S=(A/S,\iota,\lambda,\alpha)$ as follows:  $A$ is an abelian scheme of relative
dimension $g$ over  $S$;
the real multiplication $\iota\colon\mathcal{O}_L \hookrightarrow {\rm End}_S(A)$ is a ring
homomorphism endowing $A$ with an action of $\ol$;  the map $\lambda$ is a polarization as in \cite{DP}; $\alpha$ is a rigid
$\Gamma_{00}(N)$-level structure,  that is, $\alpha\colon \mu_N
\otimes_{\ZZ} \gerd_L^{-1} \arr A$, an $\ol$-equivariant
closed immersion of group schemes. Since $p$ is unramified in $L$, the existence of $\lambda$ is equivalent to $\Lie(A)$ being a locally free $\ol\otimes\calO_S$-module.

 Let
$Y/W(\kappa)$ be the Hilbert modular scheme classifying
$(\underline{A}/S, H)$, where $\underline{A}$ is as above and  $H$
is a finite flat  $\mathcal{O}_L$-subgroup scheme of $A[p]$
of rank $p^g$ which is isotropic with respect  to the $\lambda$-Weil
pairing.

Let $\gerp$ be an ideal of $\ol$ dividing $p$.  Let $Y(\gerp)$ be the scheme which classifies $(\underline{A}/S, D)$, where $\uA$ is as above, and  $D\subseteq A[\gerp]$ is an
isotropic $\ol$-invariant finite-flat subgroup scheme of rank $p^{f_\gerp}$.

If $(\uA,H)$ is classified by $Y$, we can decompose 
\[
H=\oplus_{\gerp|p} H_\gerp,
\]
where the decomposition is induced from $\ol \otimes \ZZ_p \cong \Pi_{\gerp | p}\calO_{L,\gerp}$. It follows that each $(\uA,H_\gerp)$ is classified by $Y(\gerp)$, and we have $Y=\Pi_X Y(\gerp)$.

Let
$\pi\colon Y \rightarrow X$ be the natural morphism which on points
is $(\underline{A},H) \mapsto \underline{A}$. Let $\overline{X}, \mathfrak{X}, \mathfrak{X}_{\rm rig}$  be,
respectively,  the special fibre of $X$, the completion of $X$ along
$\overline{X}$, and the rigid analytic space associated to
$\mathfrak{X}$ in the sense of Raynaud. We use similar notation
$\overline{Y}, \mathfrak{Y}, \mathfrak{Y}_{\rm rig}$ for $Y$, and let
$\pi$ denote any of the induced morphisms.
For a point $P \in
\mathfrak{X}_{\rm rig}$, we denote by $\overline{P}={\rm sp}(P)$ its
specialization in $\overline{X}$, and similarly for $Y$.

Let $\gerp|p$ be a a prime ideal of $\ol$. Define 
$w_\gerp\colon Y \rightarrow Y$ to be the automorphism 
\[
(\underline{A},\oplus_{\gerp^\prime|p} H_{\gerp^\prime})
\mapsto (\underline{A}/H_{\gerp},A[\gerp]/H_{\gerp} \oplus \oplus_{\gerp^\prime\neq \gerp} \overline{H_{\gerp^\prime}}).
\]
Let $w\colon Y \arr Y$ be the automorphism $(\uA,H) \arr (\uA/H, \uA[p]/H)$. The morphisms $w_\gerp$ commute with each other, and their composition is $w$. We denote by the same notation, the analytifications $w,w_\gerp: \Yrig \arr \Yrig$.


In the following, we recall stratifications on $\Xbar$ and $\Ybar$ studied, respectively, in \cite{GorenOort} and \cite{GorenKassaei}. First, recall that to any $(\uA,H)$ in $Y$, one  can associate the data $(f\colon\uA
\arr \uB)$, where  $\uB =\uA/H$ is a polarized abelian scheme
with real multiplication and $\Gamma_{00}(N)$-structure, and $f:\uA \arr \uB$ is the natural projection, which is an $\ol$-isogeny killed by $p$ and of degree $p^g$. See Lemma 2.12 of \cite{GorenKassaei}.

 Let $k\supseteq \kappa$ be a field. For a closed point $\Qbar=(\uA,H)$ of $\Ybar$ defined over $k$, corresponding to $f\colon \uA \arr \uB$ as above, there is a unique $\ol$-isogeny $f^t\colon
\uB \arr \uA$ such that $f^t\circ f = [p_A]$ and $f\circ f^t = [p_B]$.
There are induced morphisms
\begin{align}
\bigoplus_{\beta\in \BB}\Lie(f)_\beta & \colon \bigoplus_{\beta\in
\BB}\Lie(\uA)_\beta \Arr \bigoplus_{\beta\in \BB}\Lie(\uB)_\beta,
\\\nonumber
 \bigoplus_{\beta\in \BB}\Lie(f^t)_\beta & \colon \bigoplus_{\beta\in
\BB}\Lie(\uB)_\beta \Arr \bigoplus_{\beta\in \BB}\Lie(\uA)_\beta.\label{Equation: Lie(A)}
\end{align}
We note that since $\Lie(\uA)$ is a free $\ol\otimes k$-module,
$\Lie(\uA)_\beta$ is a one dimensional $k$-vector space.  We let
\begin{align}
\varphi(\Qbar) & = \varphi(\uA, H)  = \{ \beta\in \BB\colon
\Lie(f)_{\sigma^{-1} \circ \beta} = 0\}, \notag\\
\eta(\Qbar) & = \eta(\uA, H)  = \{ \beta\in \BB\colon \Lie(f^t)_\beta =
0\}, \\
I(\Qbar) & = I(\uA, H) = \ell(\varphi(f)) \cap \eta(f) = \{ \beta\in
\BB\colon \Lie(f)_\beta = \Lie(f^t)_\beta= 0\}.\notag
\end{align}

Let $\pe$ be a pair of subsets of $\BB$. We say that $(\varphi,
\eta)$ is an \emph{admissible pair} if $\ell(\varphi^c) \subseteq \eta$.
There are $3^g$ admissible pairs and for any   $k$-rational point $\Qbar$ of $\Ybar$ the pair $(\varphi(\Qbar), \eta(\Qbar))$ is admissible.
Given another admissible pair $(\varphi', \eta')$, we say that
$(\varphi', \eta') \geq \pe$, if  $\varphi' \supseteq
\varphi$ and $\eta' \supseteq \eta$.

 In \cite{Stamm}  it was shown that for a closed point $\Qbar$ of $\Ybar$, defined
over a field $k \supseteq \kappa$, there is an isomorphism

\begin{equation}
\label{equation:local deformation ring at Q bar -- arithmetic version}
 \widehat{\calO}_{Y, \Qbar} \cong W(k) [\![ \{x_\beta,
 y_\beta: \beta \in I(\Qbar)\}, \{z_\beta: \beta \not\in I(\Qbar)\}]\!]/(\{x_\beta y_\beta - p: \beta\in
 I\}).
\end{equation}
inducing
\begin{equation}
\label{equation: local deformation ring at Q bar}
 \widehat{\calO}_{\Ybar, \Qbar} \cong k [\![ \{x_\beta,
 y_\beta: \beta \in I(\Qbar)\}, \{z_\beta: \beta \not\in I(\Qbar)\}]\!]/(\{x_\beta y_\beta: \beta\in
 I\}).
\end{equation}
The variables can be, and will always be, chosen such that the vanishing of $x_\beta$ cuts out the locus where $\beta\in \eta$, and the vanishing of $y_\beta$ cuts out the locus where $\beta\in \ell(\varphi)$.

 Let $\Pbar$ be a closed point of $\Xbar$ corresponding to $\uA$ defined over $k$ a prefect field containing $\kappa$. The
type of $\uA$ is defined by
\begin{equation} \tau(\uA) = \{\beta \in \BB: \DD\left(\Ker(\Fr_A)\cap
\Ker(\Ver_A)\right)_\beta\neq 0\},
\end{equation}where $\DD$ denotes the contravariant Dieudonn\'e module.

The relationship between the type and the $\pe$ invariants is explained in the following Lemma proven in \cite{GorenKassaei}, Corollary 2.3.4.

\begin{lem} \label{lemma: type and phi eta} Let $\Qbar$ be a closed point of $\Ybar$, and $\Pbar=\pi(\Qbar)$. We have
\[
\varphi(\Qbar) \cap \eta(\Qbar) \subseteq \tau(\Pbar)\subseteq  (\varphi(\Qbar) \cap \eta(\Qbar)) \cup (\varphi(\Qbar)^c \cap \eta(\Qbar)^c).
\]
 
\end{lem}

\

The Ekedahl-Oort stratification on $\Xbar$ was studied in \cite{GorenOort}. Let $\tau \subseteq \BB$.  There is a locally closed subset
$W_\tau$ of $\Xbar$ with the property that a closed point $\Pbar$ of
$\Xbar$ corresponding to $\uA$ belongs to $W_\tau$ if and only if
$\tau(\uA) = \tau$.  These subsets have the following properties:\begin{enumerate}
\item The collection $\{W_\tau: \tau \subseteq \BB\}$ is a
stratification of $\Xbar$ and $Z_\tau:=\overline{W}_\tau =
\bigcup_{\tau' \supseteq \tau} W_{\tau'}$.
\item Each $W_\tau$ is non-empty, regular, and equi-dimensional
of dimension $g - \sharp\; \tau$. \item The strata $\{W_\tau\}$
intersect transversally. In fact, let $\Pbar$ be a closed
$k$-rational point of $\Xbar$. There is a choice of isomorphism
\begin{equation}\label{equation: local def ring of X at P bar} \widehat{\calO}_{X, \Pbar} \cong
W(k)[\![t_\beta: \beta \in \BB ]\!],
\end{equation}
inducing
\begin{equation}\label{equation: local def ring at P bar} \widehat{\calO}_{\Xbar, \Pbar} \cong
k[\![t_\beta: \beta \in \BB ]\!],
\end{equation}
such  that for $\tau' \subseteq \tau(\Pbar)$, $W_{\tau'}$ (and $Z_{\tau'})$ are
given in $\Spf (\widehat{\calO}_{\Xbar, \Pbar})$ be the equations
$\{t_\beta = 0: \beta \in \tau'\}$.
\end{enumerate}

We now recall the Kottwitz-Rapoport stratification on $\Ybar$ which was studied in \cite{GorenKassaei}. For an admissible pair $\pe$ there is a locally closed subset $\Wpe$ of $\Ybar$ whose closed points are those $\Qbar$ in
$\Ybar$ with invariants $\pe$. Moreover, the subset
\[ \Zpe = \bigcup_{(\varphi', \eta') \geq \pe} W_{\varphi', \eta'}\]
is the Zariski closure of $\Wpe$. The following is Theorem 2.5.2 of \cite{GorenKassaei}.

\begin{thm}\label{theorem: fund'l facts about the stratificaiton of Ybar}
Let $\pe$ be an admissible pair.
\begin{enumerate}
\item $\Wpe$ is non-empty. The
collection $\{\Wpe\}$ is a stratification of $\Ybar$ by $3^g$
strata.
\item $\Wpe$ and $\Zpe$ are equidimensional, and
\[ \dim (\Wpe) = \dim(\Zpe) = 2g - (\sharp\; \varphi + \sharp \eta).\]
\item The irreducible components of $\Ybar$ are the irreducible
components of the strata $Z_{\varphi, \ell(\varphi^c)}$ for $\varphi
\subseteq \BB$.

\item We have $w(Z_{\varphi,\eta})=Z_{r(\eta),\ell(\varphi)}$. More generally, for any prime ideal $\gerp|p$, we have 
\[
\varphi(w_\gerp(\Qbar))=(\varphi(\Qbar) \cap (\BB-\BB_\gerp)) \cup (r(\eta(\Qbar)) \cap \BB_\gerp),
\]
\[
\eta(w_\gerp(\Qbar))=(\eta(\Qbar) \cap (\BB-\BB_\gerp)) \cup (\ell(\varphi(\Qbar)) \cap \BB_\gerp).
\]

\item Let $\Qbar$ be a closed point of $\Ybar$ with invariants
$\pe$, and $I = \ell(\varphi) \cap \eta$. For an admissible pair
$\peprime$, we have $\Qbar\in Z_{\varphi^\prime,\eta^\prime}$ if and only if we have:
\[ \varphi \supseteq \varphi' \supseteq \varphi - r(I), \qquad \eta \supseteq \eta' \supseteq \eta - I.\]
In that case, write $\varphi' = \varphi - J, \eta' = \eta - K$ (so that $\ell(J) \subseteq I, K \subseteq I$ and
$\ell(J) \cap K = \emptyset$). We have:
\[
\widehat{\calO}_{Z_{\varphi', \eta'}, \Qbar} = \widehat{\calO}_{\Ybar, \Qbar}/\calI,
\]
where $\calI = \left\langle \{x_\beta: \beta \in I - K\}, \{y_\gamma: \gamma \in I - \ell(J)\} \right\rangle.$
This implies that each stratum in the stratification $\{\Zpe\}$ is
non-singular.
\end{enumerate}
\end{thm}

\begin{dfn}\label{Definition: horizontal
strata}   Let $\gert$ be an ideal of $\ol$ dividing $p$.  Let $\gert^\ast$ be defined via $\gert \gert^\ast =
p\ol$, and define $f_\gert = \sum_{\gerp\vert \gert} f_\gerp$. Let $\BB_\gert = \cup_{\gerp \vert \gert} \BB_\gerp$. The horizontal strata of $\Yrig$ are defined to be those of the form $Z_{\BB_\gert,\BB_{\gert^\ast}}$.
\end{dfn}

We recall the following result (Theorem 2.6.13 of \cite{GorenKassaei}) about the irreducible components of the strata $\Zpe$. This result rules out the existence of  irreducible components of strata which may be,  in some sense, ``isolated" without affecting the general combinatorics of the stratification. This is used in this article to prove connectedness of certain regions of $\Yrig$.

\begin{thm}\label{Theorem: C cap YF cap YV}
Let $C$ be an irreducible component of $\Zpe$. Then
\[C \cap Z_{\BB,\emptyset} \cap Z_{\emptyset,\BB} \neq
\emptyset.\]
\end{thm}

We end this section with Theorem 2.6.4 (3) of  \cite{GorenKassaei}.

\begin{thm}\label{Theorem: generic type}
On every irreducible component of $\Zpe$, the type is generically $\varphi \cap \eta$.
\end{thm}

 \subsection{$p$-adic Dissections}  Let $\CC_p$ be the completion of an algebraic closure of $\QQ_p$. It
has a valuation $\val:\CC_p \arr \QQ \cup \{\infty\}$ normalized so
that $\val(p)=1$. Define
\[
\nu(x)=\min\{\val(x),1\}.
\]

For a point $P\in\Xrig$ (respectively, $Q\in\Yrig$), we denote its specialization by $\Pbar$ (respectively, $\Qbar$). We recall the definition of valuation vectors for points on $\Xrig$ and $\Yrig$ given in \cite{GorenKassaei}. Let $P\in\Xrig$.  Let $D_P={\rm
sp}^{-1}(\Pbar)$. The parameters $t_\beta$ in (\ref{equation: local def ring of X at P bar}) are functions on $D_P$. We define $\nu_\gerX(P)=(\nu_\beta(P))_{\beta\in
\mathbb{B}}$, where the entries $\nu_\beta(P)$ are given by
\[
\nu_\beta(P)=\begin{cases} \nu(t_\beta(P)) & \beta \in
\tau(\overline{P}),\\ 0 & \beta \not\in \tau(\overline{P}).
\end{cases}
\]
The above definition
is independent of the choice of parameters as in (\ref{equation: local def ring of X at P bar}).   Similarly, for $Q\in\Yrig$, we can define
$\nu_\gerY(Q)=(\nu_\beta(Q))_{\beta\in \mathbb{B}}$, where
 \[
 \nu_\beta(Q)=\begin{cases}
1 & \beta \in \eta(\overline{Q})-I(\overline{Q})=\ell(\varphi(\Qbar))^c,\\
\nu(x_\beta(Q))&\beta \in I(\overline{Q}),\\
0&\beta \not\in \eta(\overline{Q}).
\end{cases}
\]
This
definition is independent of the choice of parameters as in  (\ref{equation: local deformation ring at Q bar}).
 \medskip

It is easy to see the following.
\begin{prop}\label{Proposition: valuations under w}
For any $Q \in \Yrig$, $\beta\in \BB_\gerp$, we have
$\nu_\beta(Q)+\nu_\beta(w_\gerp(Q))=1$. In particular, for any $\beta\in \BB$, we have $\nu_\beta(Q)+\nu_\beta(w(Q))=1$.
\end{prop}

\begin{dfn}\label{Definition: total valuation}Let $Q\in \Yrig$. We define $\nu(Q)=\sum_{\beta\in\BB} \nu_\beta(Q)$. 
\end{dfn}

This valuation has values between $0$ and $g$. By the above proposition, we have $\nu(w(Q))=g-\nu(Q)$. For any interval $I \subset [0,g]$, there is an admissible open subset $\Yrig I$ whose points are $\{Q\in \Yrig: \nu(Q) \in I\}$.

\subsubsection{The valuation hypercube}

 Let $\Theta=[0,1]^\BB$ be the unit cube in $\RR^\BB$, whose $3^g$ ``open faces'' can be encoded by vectors ${\bf a}=(a_\beta)_{\beta\in\BB}$ such that
$a_\beta\in\{0,\ast,1\}$.  The face  corresponding to ${\bf a}$ is
the set
\[
\calF_{\bf a}:=\{ {\bf v}=(v_\beta)_{\beta\in\BB}\in \Theta: v_\beta=a_\beta\ {\rm if}\
a_\beta\neq\ast,\ {\rm and}\ 0<v_\beta<1\ {\rm otherwise}\}.
\]
We define  ${\rm Star}(\calF)=\cup_{\overline{\calF^\prime} \supseteq \calF}
\calF^\prime$, where the union is over all open faces $\calF^\prime$
whose topological closure contains $\calF$. For ${\bf a}$ as above,
we define
\begin{align*}
\eta({\bf a})&=\{\beta\in\BB: a_\beta\neq 0\},\\
I({\bf a})&=\{\beta\in\BB:a_\beta=\ast\},\\
\varphi({\bf a})&=r(\eta({\bf a})^c\cup I({\bf a}))=\{\beta\in\BB:
a_{\sigma^{-1}\circ\beta}\neq 1\}.
\end{align*}

\medskip

The following is proven in \cite{GorenKassaei}. It allows us to visualize various regions of $\Yrig$ using combinatorial data on the $g$-dimensional hypercube
as well as their specializations in terms of the strata $\Wpe$. This is important in understanding the process of analytic continuation.

\begin{thm} \label{theroem: theorem of cube} There is a  one-to-one correspondence  between the open faces of $\Theta$ and the strata
$\{W_{\varphi,\eta}\}$ of $\Ybar$, given by
\[
\calF_{\bf a} \mapsto
W_{\varphi({\bf a}),\eta({\bf a})}.
\]
It has the following properties:

\begin{enumerate}

\item  $\nu_\gerY(Q) \in \calF_{\bf a}$ if and only if $\Qbar\in W_{\varphi({\bf a}),\eta({\bf a})}$.

\item $\dim(W_{\varphi({\bf a}),\eta({\bf a})})=g-\dim(\calF_{\bf a})=\sharp\,\{\beta: a_\beta\neq\ast\}$.

\item The above correspondence is order-reversing: If $\calF_{\bf a}\subseteq\overline{\calF}_{\bf b}$, then $W_{\varphi({\bf b}),\eta({\bf b})}\subseteq\overline{W_{\varphi({\bf a}),\eta({\bf a})}}$ and vice versa. In particular,
$\nu_\gerY(Q)\in {\rm Star}(\calF_{\bf a}) \Longleftrightarrow \Qbar
\in Z_{\varphi({\bf a}),\eta({\bf a})}$.

\end{enumerate}
\end{thm}

We recall a result that provides the first step of our analytic continuation process (see Proposition \ref{Proposition: canonical analytic continuation})  . For $\gerp|p$, let
\begin{align*}
\calV_\gerp&:=\{Q\in
\Yrig:\nu_\beta(Q)+p\nu_{\sigma^{-1}\circ\beta}(Q)<p\quad \forall
\beta \in \BB_\gerp\},\\
\calW_\gerp&:=\{Q\in
\Yrig:\nu_\beta(Q)+p\nu_{\sigma^{-1}\circ\beta}(Q)>p\quad \forall
\beta \in \BB_\gerp\}.
\end{align*}
In \S 5.3 of \cite{GorenKassaei}, it is shown that these are admissible open sets and the following Lemma is proved.

\begin{lem}\label{lemma: valuations under pi}
Let $\gerp|p$ and $\beta\in\BB_\gerp$. Let $Q\in\Yrig$, and
$P=\pi(Q)$.
\begin{enumerate}
\item If $Q\in\calV_\gerp$ then $\nu_\beta(P)=\nu_\beta(Q)$.
\item If $Q\in\calW_\gerp$ then $\nu_\beta(P)=p(1-\nu_{\sigma^{-1}\circ\beta}(Q))$.
\end{enumerate}
\end{lem}

Before ending this section, we recall a  useful rigid analytic Koecher principle which was proved in  Theorem 2 of \cite{Lutkebohmert} or Theorem 3.5 of \cite{Bartenwerfer}. This was first applied in the context of analytic continuation of automorphic forms by Pilloni in \cite{Pilloni} to prove automatic extension of overconvergent Siegel modular forms to points of bad reduction.

\begin{prop}\label{Proposition: Koecher}
Let $\gerZ$ be an admissible formal scheme with a normal special fibre $\overline{Z}$ and rigid analytic fibre $\gerZ_{\rm rig}$. Let $W$ be a closed subset of $\overline{Z}$ of codimension bigger than $1$. Then, any rigid analytic function on $\spe^{-1}(\overline{Z}-W)$ extends to a rigid analytic function on $\gerZ_{\rm rig}$.
\end{prop}

\section{Kisin Modules and valuations}\label{Section: Kisin modules}
In the following, we show that for a point $Q=(\uA,C)$ of $\Yrig$,  the valuations $\nu_\beta(Q)$ can be read-off from the  Breuil-Kisin module of $C$.  Our references in this section are \cite{Breuil}, \cite{Kisin}, \cite{Savitt}.

 \subsection{Notation} Let $\calO_K$ be a finite extension of $W(\kappa)$ of residue field $\kappa^\prime$ and absolute ramification index $\nu_K(p)=e$. Let $\gerS_1$ denote the power series ring $\kappa^\prime[[u]]$ equipped with its canonical Frobenius morphism $\phi$ that sends $h(u) \mapsto (h(u))^p$. For a finite flat $p$-torsion group scheme $G$ over  $\calO_K$, we denote the Breuil-Kisin module associated to it by $\gerM(G)$, and its Frobenius morphism by $\Phi$.

 \subsection{Subgroups of Abelian Varietes}  Let $\uA$ correspond to a point $P$ on $\Xrig$, and assume that all points $Q\in \pi^{-1}(P)$ are  defined over $\calO_K$. Let $\gerM=\gerM(A[p])$, which is a free $\gerS_1$-module of rank $2g$ equipped with a Frobenius morphism $\Phi$. The $\ol$-action on $A$ induces an $\ol$-action on $\gerM$, and the induced $\ol\otimes W(\kappa^\prime)$-action on $\gerM$ allows us to decompose
\[
\gerM=\bigoplus_{\beta\in\BB}\  \gerM_\beta,
\]
where each $\gerM_\beta$ is a free $\gerS_1$-module of rank $2$, and $\Phi(\gerM_\beta) \subset \gerM_{\sigma\circ\beta}$. The subgroup schemes $C$ of $A[p]$ such that $(\uA,C)\in\Yrig$ are in bijection with free $\ol$-invariant Breuil-Kisin submodules $\gerC$ of $\gerM(A[p])$ of rank $p^g$, via
\[
C \mapsto \gerC:=\gerM(A[p]/C).
\]
The induced action of $\ol \otimes W(\kappa^\prime)$ on $\gerC$ produces a decomposition
\[
\gerC=\bigoplus_{\beta\in\BB} \gerC_\beta,
\]
where each $\gerC_\beta$ is a rank-one free submodule of $\gerM_\beta$. Let ${\bf e}_\beta$ denote a generator of $\gerC_\beta$, and assume
\[
\Phi({\bf e}_\beta)=a_\beta u^{r_\beta} {\bf e}_{\sigma\circ\beta},
\]
where $a_\beta$ is a unit in $\gerS_1$.

\begin{prop}\label{Proposition: valuations via Kisin modules} With the above notation, we have $\nu_\beta(\uA,C)=r_\beta/e$.
\end{prop}

\begin{proof}

Let $Q=(\uA,C)$, $G=\uA[p]/C$, and $\gerM=\gerM(G)$, as above. We can write $G=\oplus_{\gerp|p} G[\gerp]$. Each $G[\gerp]$ is a Raynaud $\kappa_\gerp$-vector space scheme of dimension $1$. We have $\gerM(G[\gerp])=\oplus_{\beta \in \BB_\gerp} \gerM_\beta$.

By proposition 2.5 of \cite{Savitt}, the group scheme associated to $\gerM(G[\gerp])$, i.e., $G[\gerp]$, has coordinate ring isomorphic to $\calO_K[T_\beta: \beta \in \BB_\gerp]/\calI$, where $\calI$ is an ideal of the form $\langle T_\beta^p-d_\beta  T_{\sigma\circ\beta}: \beta \in \BB_\gerp\rangle$, where  $d_\beta \in \calO_K$, and $\nu_K(d_\beta)=e-r_\beta$. If we decompose $\omega_G=\oplus (\omega_G)_\beta$ according to the $\calO_L \otimes W(\kappa)$-action, we deduce that for all $\beta \in \BB$, $(\omega_G)_\beta$ is isomorphic to $\calO_K/d_\beta\calO_K$ as an $\calO_K$-module.

Now consider $(f^t)^*_\beta: (\omega_{\uA})_\beta \arr (\omega_{\uA/C})_\beta$ which is dual to $\Lie(f^t)_\beta$ defined in Equation \ref{Equation: Lie(A)}. The $\calO_K$-module $(\omega_G)_\beta$ can be identified with $(\omega_{\uA/C})_\beta/(f^t)^*_\beta(\omega_{\uA})_\beta$. Proof of Theorem 2.4.1. of \cite{GorenKassaei} shows that, if $\beta\in I(\Qbar)$,  this is isomorphic to $\calO_K/y_\beta(Q)\calO_K$, where $y_\beta$ appears in Equation \ref{equation:local deformation ring at Q bar -- arithmetic version}. It follows that $e\nu_\beta(Q)=e\nu(x_\beta(Q))=e-\nu_K(y_\beta(Q))=e-\nu_K(d_\beta)=r_\beta$.

It remains to treat the case $\beta \not \in I(\Qbar)=\ell(\varphi(\Qbar)) \cap \eta(\Qbar)$. If $\beta \not \in \ell(\varphi(\Qbar))$, by definition, $(f^*)_\beta$ is an isomorphism, and hence $r_\beta=e-\nu_K(d_\beta)=e$. It also follows from definition of valuations that $\nu_\beta(Q)=1$. The other remaining case follows similarly.
 \end{proof}

We fix such a subgroup scheme  $C$ of $A[p]$.  Let ${\bf w}_\beta\in \gerM_\beta$ be such that $\{{\bf e}_\beta,{\bf w}_\beta\}$ is  a basis for $\gerM_\beta$.  Noting that $\gerM(C)=\gerM((\uA/C)[p])/\gerM(\uA[p]/C)$, we can apply Proposition \ref{Proposition: valuations via Kisin modules} to $(\uA/C, A[p]/C)$ and invoke Proposition \ref{Proposition: valuations under w} to deduce that, for all $\beta \in \BB$, we have
\[
\Phi({\bf w}_\beta)=h_\beta {\bf e}_{\sigma\circ\beta} + b_\beta u^{e-r_\beta}{\bf w}_{\sigma\circ\beta},
\]
where $b_\beta \in \gerS_1^\times$, $h_\beta\in \gerS_1$, and $r_\beta=e\nu_\beta(\uA,C)$. If $(\uA,D)$ is  another point on $\Yrig$, it will also be defined over $\calO_K$ by our assumption, and we can write
\[
\gerD:=\gerM(\uA[p]/D)=\bigoplus_\beta \gerD_\beta \subset \bigoplus_\beta \gerM_\beta.
\]
By rescaling and modifying the choice of ${\bf w}_\beta$ if
necessary, we can write  $\gerD_\beta=\gerS_1({\bf
e}_\beta+\mu_\beta{\bf w}_\beta)$ for some $0 \neq \mu_\beta \in
\gerS_1$. If we set $s_\beta=e\nu_\beta(\uA,D)$, Proposition
\ref{Proposition: valuations via Kisin modules} implies
\[
\Phi({\bf e}_\beta+\mu_\beta{\bf w}_\beta)=c_\beta u^{s_\beta}({\bf e}_{\sigma\circ\beta}+\mu_{\sigma\circ\beta}{\bf w}_{\sigma\circ\beta}),
\]
where $c_\beta\in\gerS_1^\times$. From this, one can extract relations among $\mu_\beta$'s. We summarize these in the following proposition.

\begin{prop}\label{Proposition: Kisin Modules}
Let $(\uA,C)$ be a point in $\Yrig$ defined over $\calO_K$ with $\nu_\beta(\uA,C)=r_\beta/e$ for $\beta\in\BB$, where each $r_\beta$ is an integer between $0$ and $e$. There is a basis $\{{\bf e}_\beta,{\bf w}_\beta\}_{\beta\in\BB}$ for $\gerM(\uA[p])$ such that
\[
\Phi({\bf e}_\beta)=a_\beta u^{r_\beta} {\bf e}_{\sigma\circ\beta},
\]
\[
\Phi({\bf w}_\beta)=h_\beta {\bf e}_{\sigma\circ\beta} + b_\beta u^{e-r_\beta}{\bf w}_{\sigma\circ\beta},
\]
where $a_\beta,b_\beta\in \gerS_1^\times$, and $h_\beta\in\gerS_1$. All other points $(\uA,D)$ in $\Yrig$, if all defined over $\calO_K$, are in bijection with choices of $\{\mu_\beta \in \gerS_1-\{0\}: \beta \in \BB\}$ satisfying
\[
a_\beta u^{r_\beta}+\mu_\beta^p h_\beta=c_\beta u^{s_\beta},
\]
\[
b_\beta u^{e-r_\beta}\mu_\beta^p=c_\beta u^{s_\beta}\mu_{\sigma\circ\beta},
\]
for some choice of integers $0\leq s_\beta\leq e$, and $c_\beta \in \gerS_1^\times$,    for $\beta \in \BB$, via 
\[
(\uA,D) \mapsto \gerM(\uA[p]/D)=\oplus_{\beta \in \BB}\ \gerS_1({\bf e}_\beta+\mu_\beta{\bf w}_\beta).
\]
 Under this bijection, we have $\nu_\beta(\uA,D)=s_\beta/e$.
\end{prop}

The following results will be useful later on.

\begin{lem}\label{Lemma: valuation calculation} Let $\beta \in \BB_\gerp$. Let $m_\beta$ denote the $u$-adic valuation of $\mu_\beta \in \gerS_1-\{0\}$. We have 
\[
(p^{f_\gerp}-1)m_\beta=\sum_{i=0}^{f_\gerp-1} p^{f_\gerp-1-i} (r_{\sigma^i\circ\beta} + s_{\sigma^i\circ\beta} -e).
\]
\end{lem}
 
 \begin{proof}
This follows by a simple calculation from Proposition \ref{Proposition: Kisin Modules}.
\end{proof}

\begin{cor}\label{Corollary: inequalities} Let $\beta \in \BB_\gerp$. With the above notation, we have
\[
\sum_{i=0}^{f_\gerp-1} p^{f_\gerp-1-i} (\nu_{\sigma^i\circ\beta}(\uA,C)) \geq \sum_{i=0}^{f_\gerp-1} p^{f_\gerp-1-i} (1-\nu_{\sigma^i\circ\beta}(\uA,D)).
\]
\end{cor}
\begin{proof} Let $m_\beta$ denote the $u$-adic valuation of $\mu_\beta \in \gerS_1-\{0\}$. The result follows from Lemma \ref{Lemma: valuation calculation}, and the fact that $m_\beta \geq 0$.
\end{proof}

\begin{rmk} These results can be stated in terms of the notion of {\it degree} of a finite flat group schemes as studied in \cite{Fargues}. In fact, for every finite flat group scheme $C$ corresponding to a point  $Q\in \Yrig$, this notion can be refined to define partial degrees ${\deg}_\beta(C) \in [0,1]\cap \QQ$, for every $\beta \in \BB$  (See \cite[Def 3.6]{Tian}). One can then show that $\nu_\beta(Q)=1-{\deg}_\beta(C)$. Cast in this langauge, Corollary \ref{Corollary: inequalities} has been proven in both \cite{Tian} and \cite{PilloniStroh}.
\end{rmk}

\begin{lem}\label{Lemma: Diedonne module via Kisin module}  If $0<\nu_\beta(\uA,C)<1$, then $h_\beta \in \gerS_1^\times$.
\end{lem}

\begin{proof} Let $\gerM=\oplus_{\beta\in\BB} \gerM_\beta$ denote $\gerM(\uA[p])$. Let $\DD=\oplus_{\beta\in\BB} \DD_\beta$ denote the Dieudonn\'e module
 of the reduction of $\uA[p]$, and $F$ be the Frobenius.  Then $\DD_\beta$ is isomorphic to the shift by $\sigma^{-1}$ of $\gerM_\beta/u\gerM_\beta$, where
   $F$ matches with $(\Phi\ {\rm mod}\ u)$ (see, for instance, Example 12.2.3 of \cite{BrinonConrad} ). If $h_\beta$ is not a unit, then, using
    Proposition  \ref{Proposition: Kisin Modules}, we see that $F_{\sigma\circ\beta}: \DD_{\sigma\circ\beta} \arr \DD_{\sigma^2\circ\beta}$ must be
     the zero morphism which is impossible since $\Ker(F_{\sigma\circ\beta})$ is $1$-dimensional.

\end{proof}
\section{Overconvergent Hilbert Modular Forms}\label{Section: overconvergent}


Let $X_{\QQ_\kappa}$, $Y_{\QQ_\kappa}$ denote, respectively, the base extensions of $X$, $Y$ to $\QQ_\kappa$, as well as the rigid analytic varieties associated to  them.

Let $\tilde{X}$, $\tilde{Y}$, $\tilde{\overline{X}}$, $\tilde{\overline{Y}}$ denote, respectively, toroidal compactifications of $X$, $Y$, $\Xbar$, $\Ybar$ based on a common fixed choice of rational polyhedral cone decompositions (one for each representative of $cl^+(F)$). Let $\tilde{\gerX}$, $\tilde{\gerY}$ denote the completions of $\tilde{X}$, $\tilde{Y}$ along their special fibres, and $\tXrig$, $\tYrig$ their associated rigid analytic varieties in the sense of Raynaud. We have inclusions $\Yrig \subset Y_{\QQ_\kappa} \subset \tYrig$, and similarly for $X$.

For any admissible $(\varphi,\eta)$,  define $\tilde{Z}_{\varphi,\eta}$ to be the Zariski closure of $Z_{\varphi,\eta}$ in $\tilde{\overline{Y}}$. Define 
\[
\tilde{W}_{\varphi,\eta}=\tilde{Z}_{\varphi,\eta} -  \bigcup_{(\varphi^\prime,\eta^\prime)>(\varphi,\eta)}{Z_{\varphi^\prime,\eta^\prime}}.
\]
Let $I\subseteq [0,g]$ be an interval. We define 
\[
\tYrig I=\Yrig I \cup \bigcup_{\gert|p,f_{\gert^\ast} \in I}  \spe^{-1}(\tilde{W}_{\BB_\gert,\BB_{\gert^\ast}}),
\]
where, $\BB_\gert$, $\BB_{\gert^\ast}$, $f_{\gert^\ast}$ are given in Definition \ref{Definition: horizontal strata}, and $\Yrig I$ is introduced after Definition \ref{Definition: total valuation}. We also define $Y_{\QQ_\kappa} I=\tYrig I \cap Y_{\QQ_\kappa}$. For any open subdomain $\calW$ of $\tYrig$, we define $\calW I=\calW \cap \tYrig I$. A similar notation is used for $Y_{\QQ_\kappa}$.

We also define auxiliary modular varieties as follows. Let $\gerp|p$ be a prime ideal of $\ol$.  Let $Y^{\gerp}_{\QQ_\kappa}$ be the scheme which classifies triples $(\uA,C,D)$ over $\QQ_\kappa$-schemes, where $(\uA,C)$ is classified by $Y_{\QQ_\kappa}$, $D$ is classified by $Y(\gerp)_{\QQ_\kappa}$, and $C\cap D=0$. We let $\tYrigp$ denote the toroidal compactification of $Y^\gerp_{\QQ_\kappa}$, obtained using our fixed choice of collection of cone decompositions. We use the same notation to denote the associated rigid analytic variety.  There are  two finite flat morphisms
\[
\pi_{1,\gerp},\pi_{2,\gerp}:\tYrigp \arr \tYrig,
\]
defined by
$\pi_{1,\gerp}(\uA,C,D)=(\uA,C)$ and $\pi_{2,\gerp}(\uA,C,D)=(\uA/D,C+D/D)$ on the non-cuspidal part. When $p$ is inert in $\ol$, we simply denote these maps by 
\[
\pi_1,\pi_2: \tilde{\gerY}_{\rm{rig}}^{(p)} \arr \tYrig.
\]

Let $\epsilon: \uA^{\rm univ} \arr X$ be the universal abelian scheme.
The Hodge bundle $\omega = \epsilon_\ast \Omega^1_{\uA^{\rm
univ}/X}$ is a locally free sheaf of
$\ol\otimes\calO_{X}$-modules and decomposes as a direct sum of line bundles
\[ \omega = \oplus_\beta\  \omega_\beta.
\] Let $\underline{k}=\{k_\beta:\beta \in \BB\}$ be a multiset of integers. We define
\[
\omega^{\underline{k}}=\otimes_{_\beta} \ \omega_\beta^{k_\beta}.
\]
We continue to denote by $\omega^{\underline{k}}$ the pullback of $\omega^{\underline{k}}$ to $Y$ under $\pi$. The locally free sheaf $\omega^{\underline{k}}$ extends to a locally free sheaf on $\tilde{X}$ (and similarly for $Y$), and hence to $\tXrig$ (and similarly $\tYrig$). Let $\omega^{\underline{k}}$ also denote the pullback of $\omega^{\underline{k}}$ to any $\tYrigp$ under $\pi_1$.

An overconvergent Hilbert modular form $f$ of weight $\underline{k}$ and level $\Gamma_{00}(N)$  is an element of the following space:
\[
\calM_{\underline{k}}^\dagger=\varinjlim_{r \in \QQ^{>0}}  H^0(\tYrig[0,r],\omega^{\underline{k}}).
\]
The subspace of classical modular forms is defined as
\[
\calM_{\underline{k}}=H^0(\tYrig,\omega^{\underline{k}}),
\]
where the inclusion $\calM_{\underline{k}} \arr \calM_{\underline{k}}^\dagger$ is the natural map induced by restriction.
\begin{rmk} Using Rapoport's work \cite{Rapoport}, one can show that for each $r<1$, all the natural maps
\[
H^0(\tYrig[0,r],\omega^{\underline{k}}) \arr H^0(Y_{\QQ_\kappa}[0,r],\omega^{\underline{k}}) \arr H^0(\Yrig[0,r],\omega^{\underline{k}})
\]
are isomorphisms, if $L \neq \QQ$. See Lemma 4.1.4 of  \cite{KisinLai} for details. This implies that we can define overconvergent Hilbert modular forms using any of the three rigid analytic varieties $\Yrig$, $Y_{\QQ_\kappa}$, and $\tYrig$.
\end{rmk}

There is a full Hecke algebra acting on $\calM_{\underline{k}}^\dagger$ preserving $\calM_{\underline{k}}$. In \S \ref{Section: U_p}, we will briefly recall the geometric construction of the $U$-operators at prime ideals dividing $p$.

\section{Analytic continuation}

In this section, we show that every overconvergent Hilbert modular form of finite slope extends automatically to a  region of $\Yrig$ which is large enough for our arguments.

\subsection{The $U_\gerp$ operators} \label{Section: U_p}

Recall  the standard construction of overconvergent correspondences explained in Definition 2.19 of \cite{Kassaeiclassicality}: If $\gerp |p$, and we have  $\pi_{1,\gerp}^{-1}({\calU_1})\subset \pi_{2,\gerp}^{-1}(\calU_2)$ for admissible opens $\calU_1,\calU_2$ of $\tYrig$, we can define an operator
\[
U_{\gerp}: \omega^{\underline{k}}(\calU_2)\arr \omega^{\underline{k}}(\calU_1),
\]
via
\[
U_{\gerp}(f)=(1/p^{f_\gerp}) (\pi_{1,\gerp})_{*}(res({\rm pr}^*\pi_{2,\gerp}^*(f))),
\]
where $res$ is restriction from $\pi_{2,\gerp}^{-1}(\calU_2)$ to $\pi_{1,\gerp}^{-1}({\calU_1})$, and ${\rm pr}^*: \pi_{2,\gerp}^* \omega^{\underline{k}} \arr \pi_{1,\gerp}^* \omega^{\underline{k}} $ is a morphism of sheaves on $\tYrigp$, which at $(\uA,C,D)$ is induced by ${\rm pr}^*: \Omega_{A/D} \arr \Omega_A$, the pullback map of differentials under the natural projection ${\rm pr}: A \arr A/D$.

\begin{dfn}
The $U_\gerp$ operator on $\calM_{\underline{k}}^\dagger$ is obtained as above, using the fact that, by Lemma \ref{lemma: valuations under pi}, for any $0<r<1$, we have
\[
\pi_{1,\gerp}^{-1}(\tYrig[0,r]) \subset \pi_{2,\gerp}^{-1}(\tYrig[0,r)).
\]
The $U_\gerp$ operators commute with each other and their composition is called $U_p$.
\end{dfn}

 We also define, for any $\calU \subset \tYrig$, and any $\gerp|p$,
\[
w_\gerp:\omega^{\underline{k}}(\calU) \arr \omega^{\underline{k}}(w_\gerp(\calU)),
\]
by $w_\gerp(f)={\rm pr}^* w_\gerp^* (f)$. We can similarly define 
\[
w:\omega^{\underline{k}}(\calU) \arr \omega^{\underline{k}}(w(\calU)).
\]

For $Q=(\uA,H) \in \Yrig$, we can decompose $H=\oplus_{\gerp|p} H_\gerp$. For a fixed $\gerp_0|p$, we define
\[
\Sib_{\gerp_0}(Q):=\{(\uA,H^\prime)\in \Yrig: H^\prime_{\gerp}=H_\gerp \iff \gerp\neq \gerp_0\},
\]
\[
\Sib(Q):=\pi^{-1}(\pi(Q))-\{Q\}.
\]
Elements of $\Sib(Q)$ are points of $\Yrig$ corresponding to all $(\uA,H^\prime)$, in which $H^\prime$ is different from $H$. From definition, we have  $U_\gerp(Q)=(1/p^{f_\gerp}) \sum_{Q^\prime \in \Sib_\gerp(Q)} w_\gerp(Q^\prime)$ as a correspondence on $\tYrig$. 

We recall an analytic continuation result proved essentially in \cite{GorenKassaei}.

\begin{prop}\label{Proposition: canonical analytic continuation} Let $f \in \calM_{\underline{k}}^\dagger$ be a $U_p$-eigenform with eigenvalue $a_p\neq 0$. Then, $f$ can be extended to a section of $\omega^{\underline{k}}$ on $\calV:=\{Q \in \tYrig: \nu_\beta(Q)+p\nu_{\sigma\circ\beta}(Q)<p,\  \forall \beta \in \BB\}$.
\end{prop}

\begin{proof} Assume $f\in H^0(\tYrig[0,r],\omega^{\underline{k}})$. Let $\calV(s)=\{Q \in \tYrig: \nu_\beta(Q)+p\nu_{\sigma\circ\beta}(Q)\leq s,\  \forall \beta \in \BB\}$. It is enough to show that for any $s < p$, we can extend $f$ to $\calV(s)$. Let $M\geq 0$ be an integer  such that $\calV(p^{-M}s) \subset \tYrig[0,r]$. We consider $f$ as a section of $\omega^{\underline{k}}$ over $\calV(p^{-M}s)$ by restriction. Consider the  sequence of admissible opens in $\tYrig$
\[
\calV(p^{-M}s) \subset \calV(p^{-M+1}s)  \subset \cdots \subset \calV(s).
\]
By Proposition 3.1 of \cite{Kassaeiclassicality}, to extend $f$ from $\calV(p^{-M}s)$ to $\calV(s)$, it is enough to show that for all $1\leq i\leq M$, we have $\pi_1^{-1} (\calV(p^{-M+i}s)) \subset \pi_2^{-1} (\calV(p^{-M+i-1}s))$. This follows immediately from Lemma \ref{lemma: valuations under pi} and Proposition \ref{Proposition: valuations under w}.
\end{proof}

\subsection{The locus $|\tau|\leq 1$.} \label{Section: tau}By properties of the stratification $\{W_\tau\}$ of $\Xbar$ recalled in \S \ref{section: stratifications}, the union 
\[
W_1:=\bigcup_{|\tau \cap \BB_\gerp|> 1, \exists\gerp}W_\tau
\]
 is a closed subscheme of $\Xbar$, and, in fact, a closed subscheme of $\tilde{\Xbar}$. We define
\[
\tilde{\Xbar}^{|\tau| \leq 1}:=\spe^{-1}(\tilde{\Xbar}-W_1),
\]
\[
\tilde{\Ybar}^{|\tau| \leq 1}:=\pi^{-1}(\tilde{\Xbar}^{|\tau| \leq 1}).
\]
We also define $\tXrigtau:=\spe^{-1}(\tilde{\Xbar}^{|\tau| \leq 1})$ and $\tYrigtau:=\spe^{-1}(\tilde{\Ybar}^{|\tau| \leq 1})=\pi^{-1}(\tXrigtau).$ These are admissible open subsets of $\tXrig$ and $\tYrig$, respectively.   We will study the dynamics of the Hecke correspondences $U_\gerp$ on $\tYrigtau$.

\subsection{Analytic Continuation} \label{inert assumptiom} From this point on, to keep the notation less cluttered, and the presentation clearer, we will assume that $p$ is inert in $\ol$. In \S \ref{Section: Appendix}, we will indicate how to extend these arguments to the general case where $p$ is unramified in $\ol$. The passage from the inert case to the general case is mostly of combinatorial nature.
If $p$ is inert in $\ol$, then the notation simplifies as follows: 
\[
\Sib(Q)=\Sib_{(p)}(Q),\ 
U_p=U_{(p)},\ 
W_1=\bigcup_{|\tau|>1} W_\tau,\ \pi_1:=\pi_{1,(p)}, \pi_2:=\pi_{2,(p)}, w=w_{(p)}.
\]

First, we prove a lemma regarding valuations of points in $\tYrigtau$.

\begin{lem}\label{Lemma: type one valuations}
Let $Q \in \tYrigtau$ be of non-ordinary reduction. Let $\tau(\pi(Q))=\{\beta_0\}$. There exists a unique $0\leq m \leq g-1$, such that
\[
\nu_\beta(Q)=\begin{cases}
1 & \beta=\sigma^j\circ\beta_0\ \ \ 0\leq j < m\\
\in [0,1) &\beta=\sigma^{m}\circ\beta_0\\
0&\beta=\sigma^j\circ\beta_0\ \ \ m < j \leq g-1.
\end{cases}
\]
In particular, for such points, $\nu(Q)$ determines $\nu_\beta(Q)$ for all $\beta\in\BB$, by simply choosing $m$ to be the integral part of $\nu(Q)$. Furthermore, we have $I(\Qbar)=\{\sigma^m \circ \beta_0\}$ if $\nu(Q)\not\in\ZZ$, and $I(\Qbar)=\emptyset$ otherwise.
\end{lem}

\begin{proof}  

Assume that for some $0 \leq r \leq g-1$, we have  $\nu_{\sigma^r\circ\beta_0}(Q)\neq 1$. We show that $\nu_{\sigma^s\circ\beta_0}(Q)=0$ for all $r<s<g$. If $r=g-1$, then there is nothing to prove. Otherwise, using part (1) of Theorem \ref{theroem: theorem of cube}, it follows that $\sigma^{r+1}\circ\beta_0 \in \varphi(\Qbar)$. Since $\sigma^{r+1}\circ\beta_0 \not\in \tau(\pi(\Qbar))$, Lemma
\ref{lemma: type and phi eta} implies that $\sigma^{r+1} \circ \beta_0 \not\in \eta(\Qbar)$. This, in turn, implies that $\nu_{\sigma^{r+1} \circ \beta_0}(Q)=0$. We can now replace $r$ with $r+1$, and repeat the same argument until we arrive at $r=g-1$ and the claim is proved. Since $Q$ has non-ordinary reduction, there is $0 \leq r \leq g-1$, such that $\nu_{\sigma^{r}\circ\beta}(Q) \neq 1$. Taking $m$ to be the smallest such $r$ proves the first part of the Lemma.

If we write $\nu(Q)=[\nu(Q)]+\{\nu(Q)\}$ as the sum of its integer and fractional parts,  the valuation vector of $Q$ can be determined as in the statement of the lemma with $m=[\nu(Q)]$, and with $\nu_{\sigma^m\circ\beta_0}(Q)=\{\nu(Q)\}$. The final statement follows immediately from the definition of valuations.

\end{proof}

\begin{dfn} \label{Definition: type one valuations} For future reference, we denote the valuation vector obtained in the Lemma $v(\beta_0,m,g,x)$ where $x=v_{\sigma^m\circ\beta_0}(Q)$.
\end{dfn}
We will use the following lemma  in what follows.

\begin{lem}\label{Lemma: automatic type one} Let $Q=(\uA,C) \in \tYrig$ have a valuation vector as in  Lemma \ref{Lemma: type one valuations}, and assume that $\nu(Q) \not\in\ZZ$. Then, $Q \in \tYrigtau$.
\end{lem}

\begin{proof} Part (1) of Theorem \ref{theroem: theorem of cube} allows us to calculate the $(\varphi,\eta)$-invariants of $\Qbar$ using the valuation vector of $Q$, and observe that $\varphi(\Qbar)\cap\eta(\Qbar)=\{\beta_0\}$, and $\varphi(\Qbar)^c \cap \eta(\Qbar)^c=\emptyset$. Lemma \ref{lemma: type and phi eta} implies that $\tau(\pi(\Qbar))=\{\beta_0\}$ and the lemma follows.
\end{proof}

\begin{dfn}
Let $\calR$ be the admissible open subset of $\tYrigtau$ defined as
\[
\calR=\tYrigtau[0,g-\sum_{i=1}^{g-1} 1/p^i).
\]
\end{dfn}

\begin{prop}\label{Proposition: Analytic continuation}
Let $f$ be an overconvergent Hilbert modular form of weight $\underline{k}$, which is a $U_p$ eigenform with a nonzero eigenvalue $a_p$. Then $f$ extends analytically from $\tYrig[0,0]$ to $\calR$.
\end{prop}

\begin{proof}
By Proposition \ref{Proposition: canonical analytic continuation}, $f$ extends to $\calV \cap \tYrigtau=\tYrigtau[0,1)$.   By Proposition 3.1 of \cite{Kassaeiclassicality}, to extend $f$ further to $\calR$, it is enough to show that
\[
\pi_1^{-1}(\calR) \subset \pi_2^{-1}(\tYrigtau[0,1))
\]
Let $R=(\uA,C,D)$ be such that $Q=(\uA,C) \in \calR$. We can assume that $Q$ has non-ordinary reduction, and, hence, $\tau(\uA)=\{\beta_0\}$, for some $\beta_0\in \BB$. We want to show that $\pi_2(R)\in \tYrigtau[0,1)$. This is equivalent to showing that for every $Q^\prime \in \Sib(Q)$, we have $w(Q^\prime) \in \tYrigtau[0,1)$. Equivalently, we must show that $|\tau(w(Q^\prime))|\leq 1$ and $Q^\prime\in \tYrig(g-1,g]$. We argue that the the first assertion follows from the second one. Since $Q^\prime\in \Yrigtau$, the valuation vector of $Q^\prime$ is as in Lemma \ref{Lemma: type one valuations} for some $m$. If $Q^\prime\in \tYrig(g-1,g]$, then $m=g-1$, which implies that the valuation vector of $w(Q^\prime)$ must also be as in Lemma \ref{Lemma: type one valuations} (for $m=0$). Lemma \ref{Lemma: automatic type one}, then, implies that $|\tau(w(Q^\prime))|\leq 1$.  The second assertion, i.e.,  $Q^\prime\in \tYrig(g-1,g]$, is equivalent to   $\nu_{\sigma^{g-1}\circ\beta_0}(Q^\prime)\neq 0$ by Lemma \ref{Lemma: type one valuations}. Assume it doesn't hold. Then, writing the inequality from Corollary \ref{Corollary: inequalities} for  $\beta=\sigma^{g-1}\circ\beta_0$ and noting that $\nu_\beta(Q)\leq 1$ for all $\beta\in\BB$, we deduce
\[
p^{g-1}\nu_{\sigma^{g-1}\circ\beta_0}(Q)+\sum_{i=1}^{g-1} p^{g-1-i} \geq p^{g-1},
\]
whence, $\nu_{\sigma^{g-1}\circ\beta_0}(Q) \geq 1-\sum_{i=1}^{g-1} 1/p^i$. Lemma \ref{Lemma: type one valuations} now implies that $\nu(Q) \geq g-\sum_{i=1}^{g-1} 1/p^i$, which contradicts our choice of $Q$. This proves the claim.
\end{proof}

\begin{rmk}\label{Remark: correspondence on R} It follows immediately from the proof of Proposition \ref{Proposition: Analytic continuation} that $\pi_1^{-1}(\calR)\subset\pi_2^{-1}(\calR)$.
\end{rmk}

The following is a key result. It is more refined than necessary for the purpose of the analytic continuation in Proposition \ref{Proposition: Analytic continuation}, but it is important for the gluing process and the proof of classicality that follow.

\begin{lem}\label{Lemma: saturated} Let $Q \in \tYrigtau(g-1,g-\sum_{i=1}^{g-1} 1/p^i)$. Then
\[
\Sib(Q)\subset \tYrigtau(g-1,g-\sum_{i=1}^{g-1} 1/p^i).
\]
In other words, $\tYrigtau(g-1,g-\sum_{i=1}^{g-1} 1/p^i)$ is saturated with respect to $\pi:\tYrig \arr \tXrig$.

\end{lem}

\begin{proof} Assume $\tau(\Qbar)=\{\beta_0\}$, and let $Q^\prime \in \Sib(Q)$. We  freely use the notation from \S \ref{Section: Kisin modules} by setting $Q=(\uA,C)$ and $Q^\prime=(\uA,D)$. By Proposition \ref{Proposition: valuations via Kisin modules} and Lemma \ref{Lemma: automatic type one}, we have $r_{\sigma^i\circ\beta_0}=e$ for all $0 \leq i \leq g-2$ and $r:=r_{\sigma^{g-1}\circ\beta_0} \in (0, e(1-\sum_{i=1}^{g-1} 1/p^i))$. From the proof of Proposition \ref{Proposition: Analytic continuation}, we have $s_{\sigma^i\circ\beta_0}=e$ for all $0 \leq i \leq g-2$ and $s:=s_{\sigma^{g-1}\circ\beta_0} \in (0,e)$. If $s\leq r$, we are done. Assume $s>r$.

Proposition \ref{Proposition: Kisin Modules} gives us  $a_{\sigma^{g-1}\circ\beta_0} u^r+\mu_{\sigma^{g-1}\circ\beta_0}^p h_{\sigma^{g-1}\circ\beta_0}=c_{\sigma^{g-1}\circ\beta_0} u^{s}$. By Lemma  \ref {Lemma: Diedonne module via Kisin module}, we have $h_{\sigma^{g-1}\circ\beta_0}\in \gerS_1^\times$.  Therefore, comparing $u$-adic valuations, we deduce that the valuation of $\mu_{\sigma^{g-1}\circ\beta_0}$ equals $r/p$. This valuation is calculated in Lemma \ref{Lemma: valuation calculation}. It follows that
\[
s=e-(r/ep^g)-e\sum_{i=1}^{g-1} 1/p^i < e(1-\sum_{i=1}^{g-1} 1/p^i).
\]
 Therefore, $\nu(Q^\prime) \in (g-1,g-\sum_{i=1}^{g-1} 1/p^i)$, as desired.


\end{proof}

\section{The Main Result}

\subsection{The statement} Let $p>2$ be a prime number and $L$ a totally real
field in which $p$ is unramified.  Let $\SS=\{\gerp|p\}$ be the set of all prime ideals of $\ol$ dividing $p$. For every $\gerp\in\SS$, let $D_\gerp\subset Gal(\overline{\QQ}/L)$ denote a decomposition group at $\gerp$.

\begin{thm}\label{Theorem: modularity}   Let $\rho: Gal(\overline{\QQ}/L) \arr \GL_2(\calO)$ be a continuous representation, where $\calO$ is the ring of integers in a finite extension of $\QQ _p$, and $\germ$ its maximal ideal. Assume

\begin{itemize}

\item $\rho$ is unramified outside a finite set of primes,

\item For every prime $\gerp | p$, we have $\rho_{|_{D_\gerp}} \cong \alpha_\gerp \oplus \beta_\gerp$, where $\alpha_\gerp,\beta_\gerp: D_\gerp \arr \calO^\times$ are unramified characters distinct modulo $\germ$,

\item $\overline{\rho}:=(\rho\ {\rm mod}\ \germ)$ is ordinarily modular, i.e., there exists a classical Hilbert modular form $h$  of parallel weight $2$, such that ${\rho}\equiv \rho_h\ (\mathrm{mod}\  \germ)$, and $\rho_{h}$ is potentially ordinary and potentially Barsotti-Tate at every prime of $L$ dividing $p$ (see Remark \ref{Remark: ordinary lift}),

\item $\overline{\rho}:=(\rho\ {\rm mod}\ \germ)$ is absolutely irreducible when restricted to $Gal(\overline{\QQ}/L(\zeta_p))$.

\end{itemize}

Then, $\rho$ is isomorphic to $\rho_f$,  the Galois representation
associated to a Hilbert modular eigenform $f$ of weight
$(1,1,\cdots,1)$ and level $\Gamma_{00} (N)$ for some integer $N$
prime to $p$.

\end{thm}

For any $\gerp\in\SS$, let $a_{\gerp}=\alpha_\gerp(Frob_\gerp)$, and $b_{\gerp}=\beta_\gerp(Frob_\gerp)$. Then, $a_{\gerp},b_{\gerp}$ are nonzero, and  $a_{\gerp} \neq b_{\gerp}$ for all $\gerp\in \SS$. One can repeat Buzzard-Taylor's argument \cite{BuzzardTaylor} using  Gee's results on companion Hilbert modular forms \cite{Gee}, and  applying Taylor-Wiles arguments \`a la Diamond/Fujiwara \cite{Fujiwara}, to produce a collection of $2^{|\SS|}$ overconvergent Hilbert eigenforms
\[
 \{f_T: T \subset \SS\},
 \]
 of weight $(1,\cdots,1)$ with  identical Hecke eigenvalues corresponding to prime-to-$p$ ideals, and such that  $U_\gerp(f_T)=a_{\gerp} f_T$ if $\gerp \in T$, and $U_\gerp(f_T)=b_{\gerp} f_T$ if $\gerp\not\in T$, which satisfy  $\rho\cong\rho_{f_T}$ for all $T\subset \SS$. A detailed argument is written in  \S 7 of \cite{Sasaki2}, where Kisin's modularity results have been used instead. The assumption that $p$ is split in the totally real field is not necessary for the argument presented there, as long as $p\neq 2$. In what follows, we show that the existence of $\{f_T: T \subset \SS\}$ as described above, implies that they are all classical, proving the theorem.

\subsection{Analytic continuation and gluing.} In the following, we prove that certain overconvergent Hilbert modular forms are classical. This applies to finish the proof of Theorem \ref{Theorem: modularity}. As before (starting from \S \ref{inert assumptiom}), for clarity of presentation, we will write the proof in the case $p$ is inert in $L$. The general case will be discussed in \S \ref{Section: Appendix}.

\begin{thm}\label{Theorem: Main}
Assume $p$ is inert in $\ol$. Let $f_1,f_2$ be two overconvergent Hilbert modular eigenforms of level $\Gamma_{00}(N)$ and weight $\underline{k}$. Assume that for $i=1,2$, $f_i$ is a $U_p$-eigenform with eigenvalue $a_{i,p}\neq 0$, and that $a_{1,p}\neq a_{2,p}$. Furthermore, assume that the Hecke eigenvalues of $f_1$ and $f_2$ are equal at all ideals prime to $p\ol$. Then $f_1$ and $f_2$ are classical and, indeed, the two $p$-stabilizations of a Hilbert modular form of level $\Gamma_{00}(N)$.
\end{thm}

We prove the theorem in what follows. First, we use the condition on the $q$-expansions to prove an equality of sections of $\omega^{\underline{k}}$ over a certain region of $\tYrig^{(p)}$. We need a lemma. For any choice of the polarization module $(\gera,\gera^+)$, denote the corresponding connected component of $\tYrig$ by $\tilde{\gerY}_{{\rm rig},\gera}$, and that of $\Ybar$ with $\Ybar_\gera$. For any admissible open $\calW$ of $\tYrig$, set $\calW_\gera=\calW \cap \tilde{\gerY}_{{\rm rig},\gera}$. Applying this convention to $\calR$, we can define $\calR_\gera$.

\begin{lem}\label{Lemma: connectedness}
For any choice of the polarization module $(\gera,\gera^+)$, the region $\calR_\gera$ is connected.
\end{lem}

\begin{proof}  It is enough to show that $\calR_\gera \cap \Yrig$ is connected. To ease the notation, throughout this proof, we continue to denote $\calR_\gera \cap\Yrig$ with $\calR_\gera$. We first show that $\calR_\gera[0,g-1]$ is connected.  For $0 \leq m \leq g-1$, let 
\[
\Phi_m=\{ \varphi \subset \BB|   \exists \beta \in \BB: \varphi=\{\beta,\sigma^{-1}\circ\beta,\cdots,\sigma^{1+m-g}\circ\beta\} \}.
\]
and, for $0 \leq i \leq g-1$, set
\[
W^\prime(i)=\bigcup_{0 \leq m \leq i}\ \bigcup_{\varphi \in \Phi_m} Z^\prime_{\varphi,\ell(\varphi)^c} 
\]
where, $Z^\prime_{\varphi,\ell(\varphi)^c}=Z_{\varphi,\ell(\varphi)^c}$ if $m<g-1$, and $Z^\prime_{\varphi,\ell(\varphi)^c}= Z_{\varphi,\ell(\varphi)^c}-Z_{\emptyset,\BB}$ if $m=g-1$. Define
\[
W(i)=W^\prime(i) \cap {\Ybar_\gera}^{|\tau| \leq 1}.
\]
Using Theorem \ref{theroem: theorem of cube}, it is easy to see that  $\calR_\gera[0,g-1]=\spe^{-1}(W(g-1))$. The claim would follow if we show $W(g-1)$ is connected. We argue by induction. For $i=0$, $W(0)$ contains the ordinary locus in $(Z_{\BB,\emptyset})_\gera$, and, hence, is a dense subset in (the nonsingular) $(Z_{\BB,\emptyset})_\gera$. It follows that $W(0)$ is irreducible and, hence, connected. Now, assume $W(i)$ is connected. Pick any $\Qbar\in W(i)-W(i-1)$, and set $\varphi=\varphi(\Qbar)=\{\beta,\cdots,\sigma^{1+i-g}\circ\beta\}$ for some $\beta\in\BB$. Let $\varphi^\prime=\{\sigma^{i-g}\circ\beta\}\cup \varphi$. Let $C$ be an irreducible component of $Z_{\varphi,\ell(\varphi)^c}$ which contains $\Qbar$. Theorems \ref{Theorem: C cap YF cap YV} and \ref{theorem: fund'l facts about the stratificaiton of Ybar}  imply that  $C$ intersects $Z_{\varphi^\prime,\ell(\varphi^\prime)^c} \subset W^\prime(i-1)$. To show that $W(i)$ is connected, it is enough to show that $C\cap {\Ybar_\gera}^{|\tau| \leq 1}$ is connected, and $Z_{\varphi^\prime,\ell(\varphi^\prime)^c} \cap C\cap {\Ybar_\gera}^{|\tau| \leq 1}$ is nonempty. The first statement follows from Theorem \ref{Theorem: generic type}, which shows that type is generically $\varphi \cap \ell(\varphi)^c=\{\beta\}$ on $Z_{\varphi,\ell(\varphi)^c}$, and hence on $C$. For the second statement, note that $Z_{\varphi^\prime,\ell(\varphi^\prime)^c} \cap C$ is a union of irreducible components of 
\[
Z_{\varphi^\prime,\ell(\varphi^\prime)^c} \cap  Z_{\varphi,\ell(\varphi)^c}=Z_{\varphi^\prime,\ell(\varphi)^c}.
\]
By Theorem \ref{Theorem: generic type},  on every such irreducible component the type is generically  $\varphi^\prime \cap \ell(\varphi)^c=\{\beta\}$. Hence $Z_{\varphi^\prime,\ell(\varphi^\prime)^c} \cap C\cap {\Ybar_\gera}^{|\tau| \leq 1}$ is Zariski dense in $Z_{\varphi^\prime,\ell(\varphi^\prime)^c} \cap C$, and hence is nonempty. It just remains to note that the same argument works in the last step of induction (where definitions are a bit modified) as   $Z^\prime_{\varphi,\ell(\varphi)^c}$ is a dense open subset of $Z_{\varphi,\ell(\varphi)^c}$, and $Z_{\varphi,\ell(\varphi)^c}-Z^\prime_{\varphi,\ell(\varphi)^c}$ doesn't intersect $Z_{\varphi^\prime,\ell(\varphi^\prime)^c} \cap  {\Ybar_\gera}^{|\tau| \leq 1}$.

We now prove that $\calR_\gera$ is connected.  It is enough to show that $\calR_\gera[0,g-1+v]$ is connected for all rational $0<v<1-\sum_{i=1}^{g-1} 1/p^i$.  Assume this does not hold for some such $v$.  Let $\calW$ be a connected component of $\calR_\gera[0,g-1+v]$ which doesn't intersect the connected domain $\calR_a[0,g-1]$. Since $\calR_\gera[0,g-1+v]$ is quasi-compact, so will be $\calW$.  Therefore, there is $\epsilon>0$ such that $\Yrig[0,g-1+\epsilon] \cap \calW =\emptyset$. Let $Q\in \calW$ and assume $\tau(\pi(\Qbar))=\{\beta_0\}$. By Lemma \ref{Lemma: type one valuations},  $I(\Qbar)=\{\sigma^{g-1}\circ\beta_0\}$, and $\nu(Q)=g-1+\nu(x_{\sigma^{g-1}\circ\beta_0}(Q))>g-1+\epsilon$. In fact, by definition of valuations, this statement is also valid for every $Q^\prime \in \spe^{-1}(\Qbar)\cap\calW$. Let $\calA=\spe^{-1}(\Qbar)[0,g-1+v] \subset \calR_\gera$. Using the isomorphism \ref{equation:local deformation ring at Q bar -- arithmetic version}, it follows that $\spe^{-1}(\Qbar)$ is a product of $g-1$ open discs and one annulus with variable $x_{\sigma^{g-1}\circ\beta_0}$, and $\calA$ is the subset given by $\nu(x_{\sigma^{g-1}\circ\beta_0})\leq v$. Therefore, $\calA$ is connected and we must have $\calA\subset \calW$. But this implies that on $\calA$, we must have $\nu(x_{\sigma^{g-1}\circ\beta_0})>\epsilon>0$, which is clearly not true. This completes the proof of the lemma.

\end{proof}

\begin{dfn} 
We define  $\calR^\prime=\tYrigtau[0,1)$ and $\calC=\tYrigtau(g-1,g-\sum_{i=1}^{g-1} 1/p^i)$. 
\end{dfn}

\begin{lem} \label{Lemma: union}
We have  $\tYrigtau=\calR \cup w^{-1}(\calR^\prime)$.
\end{lem}

\begin{proof} The claim follows from the equality $w^{-1}(\calR^\prime)=\tYrigtau(g-1,g]$. This equality follows from Lemma \ref{Lemma: automatic type one}.
\end{proof}

The same proof shows that 

\begin{lem} \label{Lemma: intersection}
We have  $\calC = \calR \cap w^{-1}(\calR^\prime)$.
\end{lem}

\begin{rmk} In an earlier version of this article, we had performed the (forthcoming) gluing of two forms on $\calR$ and $w^{-1}(\calR)$. This argument is done more simply and concisely  in this version, by gluing forms, instead, on $\calR$ and $w^{-1}(\calR^\prime)$. The main advantage is that Proposition \ref{Proposition: Intersection} immediately implies that the two forms in question agree on the intersection $\calR \cap w^{-1}(\calR^\prime)=\calC$, whereas, we have  $\calC \subsetneq \calR \cap w^{-1}(\calR)$, and it would require more work to show that the two forms agree on this intersection.

We have applied a similar modification to the argument in the general unramified case dealt with in \S \ref{Section: Appendix}, in Lemma \ref{Lemma: induction step}, and the lead-up to it.

\end{rmk}

 By Proposition \ref{Proposition: Analytic continuation}, we know that both $f_1,f_2$ can be extended to $\calR$. Recall from Remark \ref{Remark: correspondence on R} that $\pi_1^{-1}(\calR) \subset \pi_2^{-1}(\calR)$.

\begin{prop}\label{Proposition: equality of sections} Let $f_1,f_2$ be normalized overconvergent Hilbert eigenforms as in Theorem \ref{Theorem: Main}. We have the following equality over $\pi_1^{-1}(\calR)$:
\[
\pi_1^*(a_{1,p}f_1-a_{2,p}f_2)={\rm pr}^*\pi_2^*(f_1-f_2),
\]
where ${\rm pr}^*: \pi_2^* \omega^{\underline{k}} \arr \pi_1^* \omega^{\underline{k}}$ is the natural map defined in \S \ref{Section: U_p}.

 \end{prop}

\begin{proof} In this argument, we extend scalars to  a finite extension of the base containing an $Np$-th root of unity. Let $(\gera,\gera^+)$ be a representative of a class in $cl^+(L)$.  Fix a rational polyhedral cone decomposition $\{\sigma_\alpha\}$ on $\gera_\RR$. Consider an unramified cusp $Ta_\gera=((\underline{\GG_m\! \otimes\! \gerd_L^{-1})/q^{\gera^{-1}}},<\!{\zeta}\!>)$ (underline indicates the inclusion of standard PEL structure) inducing a local chart  $(Ta_\gera)_{\sigma_\alpha}$ of the universal family over $\tYrig$ for any fixed $\sigma_\alpha$. Here, $\zeta$ denotes an $\calO_L$-generator of $(\GG_m\otimes \gerd_L^{-1})[p]$. Let $\eta^{\underline{k}}$ denote a generator of the sheaf $\omega^{\underline{k}}$ on the base of $Ta_\gera$.


 We can write $f_i((Ta_\gera)_{\sigma_\alpha})=\sum_{\xi \in (\gera^{-1})^+}{a_{i,\xi(\gera)}} q^\xi \eta^{\underline{k}}$ for $i=1,2$. The assumptions on $f_1$, $f_2$ imply that, after normalizing, we can assume $a_{1,\xi(\gera)}=a_{2,\xi(\gera)}$ for all $\xi \not \in p(\gera^{-1})^+$, and that $a_{i,\xi(p\gera)}=a_{i,p}a_{i,\xi(\gera)}$ for all $\xi\in(\gera^{-1})^+$ and $i=1,2$. Let $c \in \gera^{-1}-p\gera^{-1}$. We will calculate both sides of the equality over $(Ta^0_{\gera})_{\sigma_\alpha}=((Ta_\gera)_{\sigma_\alpha},<\!q^c\!>)$, a local chart of the universal family over $\tYrig^{(p)}$.  By $<\!q^c\!>$, we mean the subgroup induced by the  $\calO_L$-submodule generated by $q^c$.
\begin{align*}
\pi_1^*(a_{1,p}f_1-a_{2,p}f_2)((Ta^0_{\gera,j})_{\sigma_\alpha})&=(a_{1,p}f_1-a_{2,p}f_2)((Ta_\gera)_{\sigma_\alpha})\\
&=\sum_{\xi\in(\gera^{-1})^+} (a_{1,p}a_{1,\xi(\gera)}-a_{2,p}a_{2,\xi(\gera)}) q^\xi \eta^{\underline{k}}\\
&=\sum_{\xi\in(\gera^{-1})^+} (a_{1,\xi(p\gera)}-a_{2,\xi(p\gera)}) q^\xi \eta^{\underline{k}}.\\
\end{align*}
On the other hand, we can write:

\begin{align*}
{\rm pr}^*\pi_2^*(f_1-f_2)((Ta^0_{\gera})_{\sigma_\alpha})&={\rm pr}^*((f_1-f_2)((\underline{\GG_m\! \otimes\! \gerd_L^{-1})/{q}^{(p\gera)^{-1}}},<\!\overline{\zeta}\!>))_{\sigma_\alpha})\\
&=\sum_{\xi\in p^{-1}(\gera^{-1})^+} (a_{1,\xi(p\gera)}-a_{2,\xi(p\gera)}) {q}^{\xi} \eta^{\underline{k}}\\
&=\sum_{\xi\in(\gera^{-1})^+}(a_{1,\xi(p\gera)}-a_{2,\xi(p\gera)}) {q}^{\xi} \eta^{\underline{k}}.\\
\end{align*}



 To end the proof, it is enough to show that every connected component of $\pi_1^{-1}(\calR_\gera)$ intersects the unramified cuspidal locus. Since $\pi_1$ is finite and flat, every connected component of $\pi_1^{-1}(\calR_\gera)$ maps onto $\calR_\gera$ (using Lemma \ref{Lemma: connectedness}). Now, $\calR_\gera$ contains $\tYrig[0,0]_\gera$, which, in turn, contains the entire unramified cuspidal locus of $\tilde{\gerY}_{{\rm rig},\gera}$. We are done.

 \end{proof}

 By Proposition \ref{Proposition: Analytic continuation}, the section  $a_{1,p}f_1-a_{2,p}f_2$ can be extended to $\calR$ and $w(f_1-f_2)$ can be extended to  $w^{-1}(\calR^\prime)$. We claim that these sections are equal over the intersection 
\[
\calR \cap w^{-1}(\calR^\prime)=\calC=\tYrigtau(g-1,g-\sum_{i=1}^{g-1} 1/p^i).
\]

\begin{prop}\label{Proposition: Intersection} Over $\calC$, we have:
\[
a_{1,p}f_1-a_{2,p}f_2=w(f_1-f_2).
\]
\end{prop}

\begin{proof} Let $Q=(\uA,C)$ be a point in $\calC$. By Lemma \ref{Lemma: saturated}, we have $\Sib(Q) \subset \calC$.  Hence, for any distinct pair $Q_1=(\uA,C_1)$ and $Q_2=(\uA,C_2)$ of points belonging to $\pi^{-1}(\uA)$, we have $(\uA,C_1,C_2) \in \pi_1^{-1}(\calR)$. By Proposition \ref{Proposition: equality of sections}, we can write
\[
(a_{1,p}f_1-a_{2,p}f_2)(\uA,C_1)={\rm pr}^*(f_1-f_2)(\uA/C_2,\uA[p]/C_2)=w(f_1-f_2)(\uA,C_2).
\]
Writing this equation for all possible distinct pairs as above, we readily obtain our desired equality:
\[
(a_{1,p}f_1-a_{2,p}f_2)(\uA,C)=w(f_1-f_2)(\uA,C).
\]
\end{proof}

\begin{rmk}
Note that it is important for the above argument that we are working over $\calC=\tYrigtau(g-1,g-\sum_{i=1}^{g-1} 1/p^i)$ which is a $\pi$-saturated region; no region inside $\tYrigtau[0,g-1]$ has the same property. While in the above argument we haven't used the full force of this property, we will do so below (Corollary \ref{Corollary: average}). This necessitates our chosen extent of the analytic continuation of $f_1,f_2$.
\end{rmk}

\begin{cor}\label{Corollary: gluing}The section $a_{1,p}f_1-a_{2,p}f_2$ extends to $\tYrigtau$. 
\end{cor}

\begin{proof} By the arguments above, the section $a_{1,p}f_1-a_{2,p}f_2$ on $\calR$ glues to $w(f_1-f_2)$ on $w^{-1}(\calR^\prime)$. Hence, $a_{1,p}f_1-a_{2,p}f_2$ extends analytically to $\calR \cup w^{-1}(\calR^\prime)$ which equals $\tYrigtau$ by Lemma \ref{Lemma: union}.
\end{proof}

\begin{cor}\label{Corollary: average} Let $Q=(\uA,C) \in \calC$. Let $(\uA,C^\prime)\in\Sib(Q)$. We have
\[
(a_{1,p}f_1-a_{2,p}f_2) (\uA,C)=(a_{1,p}f_1-a_{2,p}f_2) (\uA,C^\prime)
\]
\end{cor}
\begin{proof} Looking at the proof of Proposition \ref{Proposition: Intersection}, we find that if $(\uA,D)$ is another point in $\Sib(Q)$, we can write
\[
(a_{1,p}f_1-a_{2,p}f_2) (\uA,C)=w(f_1-f_2)(\uA,D)=(a_{1,p}f_1-a_{2,p}f_2) (\uA,C^\prime).
\]

\end{proof}

 We can now finish the proof of Theorem \ref{Theorem: Main}. We first show that $a_{1,p}f_1-a_{2,p}f_2$ extends from $\tYrigtau$ to $\tYrig$. Since the morphism
 \[
 \pi:\tYrigtau \arr \tXrigtau
 \]
 is finite and flat of rank $p^g+1$, there is a trace map
 \[
 \pi_\ast: H^0(\tYrigtau,\omega^{\underline{k}})=H^0(\tYrigtau,\pi^*\omega^{\underline{k}}) \arr H^0(\tXrigtau,\omega^{\underline{k}}),
 \]
satisfying $\pi_*\pi^*(s)=(p^g+1)s$, for all $s \in  H^0(\tXrigtau,\omega^{\underline{k}})$. We define 
 \[
 h=\pi_*(a_{1,p}f_1-a_{2,p}f_2),
 \]
 a section of $\omega^{\underline{k}}$ on $\tXrigtau$. Note that, by definition (\S \ref{Section: tau}), we have 
 \[
 \tXrigtau=\spe^{-1}({\tilde{\Xbar}}-W_1),
 \]
 and that $\codim(W_1)>1$ by results of  \cite{GorenOort}. Therefore, the rigid analytic Koecher principle implies that $h$ extends to a section of $\omega^{\underline{k}}$ on $\tXrig$, which we will still call $h$. We show that $(1/(p^g+1))\pi^*(h)$ provides an analytic continuation of $a_{1,p}f_1-a_{2,p,}f_2$ to $\tYrig$. Since all local rings in $\tYrig$ are integral and every connected component of  $\tYrig$ intersects $\calC$, it is enough to show that $(p^g+1)(a_{1,p}f_1-a_{2,p}f_2)$ agrees with $\pi^*(h)$ on $\calC$. Over $\calC$, we can write
 \[
 \pi^*(h)(\uA,C)=h(\uA)=\sum_D (a_{1,p}f_1-a_{2,p}f_2)(\uA,D)=(p^g+1)(a_{1,p}f_1-a_{2,p}f_2)(\uA,C),
 \]
where $(\uA,D)$ runs over $\pi^{-1} (\uA)$, and the last equality follows from Corollary \ref{Corollary: average}. Therefore, $a_{1,p}f_1-a_{2,p}f_2$ is classical. Proposition \ref{Proposition: Intersection} now implies that $w(f_1)-w(f_2)$ is also classical. Also, applying $w$, it follows that  $a_{1,p}w(f_1)-a_{2,p}w(f_2)$ is classical. Since $a_{1,p}\neq a_{2,p}$, we deduce that both $f_1$ and $f_2$ are classical. It also follows that $f_1$ and $f_2$ are $p$-stabilizations of $a_{1,p}f_1-a_{2,p}f_2$ which we have shown is  classical of level $\Gamma_{00}(N)$. This finishes the proof of Theorem \ref{Theorem: Main}, and, hence, of Theorem \ref{Theorem: modularity}.

\section{The general unramified case}\label{Section: Appendix}

In this section, we indicate how our proofs in the inert case can be extended to the general case, where $p$ is unramified in the totally real field. The notation is as in \S \ref{section: stratifications}. Let $\SS=\{\gerp|p\}$, the set of all prime ideals of $\ol$ dividing $p$. We denote the complement of $S\subset \SS$ by $S^c$.

\begin{dfn} Let $Q\in \Yrig$. For any $\gerp \in \SS$, we define $\nu_\gerp(Q)=\sum_{\beta\in\BB_\gerp} \nu_\beta(Q)$. If $\underline{I}$ is a multiset of intervals indexed by $\SS$, $\underline{I}=\{I_\gerp\subset [0,f_\gerp]: \gerp\in\SS\}$, and $\calV \subset \Yrig$, we define an admissible open of $\Yrig$
\[
\calV{\underline I}=\{Q \in \calV: \nu_\gerp(Q) \in I_\gerp, \forall \gerp \in \SS\}.
\]
We also define
\[
\tYrig \underline{I}=\Yrig \underline{I} \cup \bigcup_{\gert} \spe^{-1}(\tilde{W}_{\BB_\gert,\BB_{\gert^\ast}}),
\]
where $\gert$ runs over all ideals $\gert|p$, such that for all prime ideals $\gerp|\gert^\ast$, we have $f_{\gerp}\in I_\gerp$, and for all $\gerp\!\!\not| \gert^\ast$, we have $0\in I_\gerp$. See Definition  \ref{Definition: horizontal strata} and \S \ref{Section: overconvergent} for the notation. 

\end{dfn}


\begin{dfn} Let $\gerp\in \SS$. We define $\II^\prime_\gerp$ to  be the interval $[0,1)$ if $f_\gerp>1$, and $[0,p/(p+1))$ if $f_\gerp=1$. We define $\II_\gerp$ to be the interval $[0,f_\gerp-\sum_{i=1}^{f_\gerp-1} 1/p^i)$ if $f_\gerp>1$, and $[0,1)$ if $f_\gerp=1$.  We also define $\II^\ast_\gerp$ to be the interval $(f_\gerp-1,f_\gerp-\sum_{i=1}^{f_\gerp-1} 1/p^i)$ if $f_\gerp>1$, and $(0,1)$ if $f_\gerp=1$. We define $\II$ (respectively, $\II^\prime$, $\II^\ast$) to be the multiset of intervals (as defined above), whose component at $\gerp$ is $\II_\gerp$ (respectively, $\II^\prime_\gerp$, $\II^\ast_\gerp$).
\end{dfn}

Using an identical argument as in Lemma \ref{Lemma: type one valuations}, we can prove
\begin{lem}\label{Lemma: type one valuations general case}
Let $Q \in \tYrigtau$. Let $\gerp\in \SS$ be such that $\tau(\pi(\Qbar)) \cap \BB_\gerp=\{\beta_0(\gerp)\}$. There is an integer  $0\leq m(\gerp) \leq f_\gerp-1$, and a rational number $0 \leq x(\gerp) <1$, such that $(\nu_\beta(Q): \beta\in\BB_\gerp)$ is equal to $v(\beta_0(\gerp),m(\gerp),f_\gerp,x(\gerp))$ (see Definition \ref{Definition: type one valuations}). Furthermore, for any $\gerp \in \SS$ such that $\tau(\pi(\Qbar)) \cap \BB_\gerp=\emptyset$, we have, either $\nu_\beta(Q)=0$ for all $\beta \in \BB_\gerp$ or $\nu_\beta(Q)=1$ for all $\beta \in \BB_\gerp$.
\end{lem}

The following lemma is an analogue of Lemma \ref{Lemma: automatic type one}, and can be proved in exactly the same way.
\begin{lem}\label{Lemma: automatic type one, split} Let $Q=(\uA,C) \in \tYrig$. Assume $\gerp\in \SS$ is a prime ideal such that the valuation vector $(\nu_\beta(Q): \beta\in\BB_\gerp)$ equals $v(\beta_0(\gerp),m(\gerp),f_\gerp,x(\gerp))$ for a rational number $0 \leq x(\gerp) <1$ (see Definition \ref{Definition: type one valuations}). If $\nu_\gerp(Q) \not\in\ZZ$, then, $\tau(\pi(\Qbar)) \cap \BB_\gerp=\{\beta_0(\gerp)\}$.
\end{lem}

Proposition \ref{Proposition: canonical analytic continuation} implies the following.

\begin{prop}\label{Proposition: canonical analytic continuation split} Any $f \in \calM^\dagger_{\underline{k}}$ which is a $U_\gerp$-eigenform with a nonzero eigenvalue for all $\gerp|p$, can be extended to $\tYrigtau \II^\prime$.
\end{prop}

Using the same method as Proposition \ref{Proposition: Analytic continuation}, we can prove

\begin{prop}\label{Proposition: analytic continuaion split} Let $\gerq\in \SS$. Let $f$ be an overconvergent Hilbert modular form of weight $\underline{\kappa}$, which is a $U_\gerq$ eigenform with a nonzero eigenvalue $a_\gerq$. Let ${\underline I}^\prime$ be a multiset of intervals such that  $I^\prime_{\gerq}=\II^\prime_{\gerq}$. Assume $f$ is defined on $\tYrigtau {\underline I}^\prime$. Then, $f$ can be analytically extended to $\tYrigtau {\underline I}$, where $\underline{I}$ is such that $I_\gerp=I^\prime_\gerp$ for $\gerp\neq\gerq$, and ${I}_{\gerq}=\II_{\gerq}$.
\end{prop}
 Proposition \ref{Proposition: canonical analytic continuation split} combined with Proposition \ref{Proposition: analytic continuaion split} imply the following Corollary.
 
\begin{cor} Let $f$ be an overconvergent Hilbert modular form of weight $\underline{\kappa}$, which is a $U_\gerp$ eigenform with a nonzero eigenvalue $a_\gerp$, for all $\gerp \in \SS$. Then, $f$ extends to $\tYrigtau {\II}$.
\end{cor}

Exactly, as in Lemma \ref{Lemma: saturated}, we can prove

\begin{lem}\label{Lemma: saturated at one prime}Let $\underline{I}$ be such that for some $\gerq \in \SS$, we have $I_\gerq=\II^\ast_\gerq$. Let $Q \in \tYrigtau \underline{I}$, then
\[
\Sib_\gerq(Q)\subset \tYrigtau \underline{I}.
\]
\end{lem}

The following corollary follows immediately.
\begin{cor}\label{Corollary: Saturated general} The region $\calC^\ast:=\tYrigtau \II^\ast$ is saturated with respect to the morphism $\pi:\tYrig \arr \tXrig$.
\end{cor}

By discussions after the statement of Theorem \ref{Theorem: modularity}, to extend the proof of this Theorem to the general case where the prime $p$ is unramified in $\ol$, it is enough to prove the following result.

\begin{thm}\label{Theorem: classical general}
Consider a collection of elements of  $\calM^\dagger_k$,
\[
\{f_T: T \subset \SS\},
\]
consisting of Hecke eigenforms with identical prime-to-$p$ Hecke eigenvalues for varying $T$. Assume, also, that the $U_\gerp$-eigenvalue of $f_T$ is $a_\gerp$, if $\gerp\in T$, and $b_\gerp$, if $\gerp \not\in T$, with $a_\gerp\neq b_\gerp$ for all $\gerp|p$. Then, all $f_T$'s are classical. Furthermore, there is a classical Hilbert modular form $h$ of weight $\underline{\kappa}$ and level $\Gamma_{00}(N)$ with the same prime-to-$p$ Hecke eigenvalues as all the $f_T$'s.
\end{thm}

The proof will be by induction. We first prove a lemma which provides the induction step. 

We begin by setting some notation. For any $\Sigma \subset \SS$, let $\HH_\Sigma$ be the Hecke algebra generated by  the Hecke operators at all ideals prime to $p$, as well as all the $U_\gerp$ operators for $\gerp \in \Sigma$. For $T\subset \Sigma$, let $\lambda_T$ denote the system of Hecke eigenvalues on $\HH_\Sigma$ which is equal to the common system of eigenvalues of $\{f_T: T \subset \SS\}$ at prime-to-$p$ ideals, and $a_\gerp$ (respectively, $b_\gerp$) at $U_\gerp$ for $\gerp \in T$ (respectively,  $\gerp \not\in T$).

Let $S\subset \SS$, and choose $\gerq \not \in S$. Define ${\underline I}_S$ to be the multiset of intervals such that $(I_S)_\gerp=[0,f_\gerp]$ for $\gerp \in S$, and $(I_S)_\gerp=\II_\gerp$ for $\gerp \not\in S$.  
Therefore, ${\underline I}_{S\cup\{\gerq\}}$ will have the same components as ${\underline I}_S$ at prime ideals $\gerp\neq\gerq$, and at $\gerq$, we have  $(I_{S\cup\{\gerq\}})_\gerq=[0,f_\gerq]$. Let $\calR_S=\tYrigtau{\underline I}_S$.  Given $\gerq \not\in S$, define ${\underline I}_{S,\gerq}$ to be the multiset of intervals such that $(I_{S,\gerq})_\gerp=(I_S)_\gerp$ for $\gerp \neq \gerq$, and $(I_{S,\gerq})_\gerq=[0,1)$. Define $\calR_{S,\gerq}=\tYrigtau{\underline I}_{S,\gerq} \subset \calR_S$.

For any polarization module $(\gera,\gera^+)$,  we define ${\calR_{S,\gera}}=\calR_S\cap \tilde{\gerY}_{{\rm rig},\gera}$.
Recall that  $\calC^\ast=\tYrigtau \II^\ast$. An argument as in Lemma \ref{Lemma: connectedness} can be used to show that $\calR_{S,\gera}$ is connected. Also, as in the proof of Lemmas \ref{Lemma: union} and  \ref{Lemma: intersection}, we can use  Lemma \ref{Lemma: automatic type one, split} to prove that 
\[
\calR_{S\cup \{\gerq\}}=\calR_S \cup w_\gerq^{-1}(\calR_{S,\gerq}),
\]
and that $\calR_S \cap w_\gerq^{-1}(\calR_{S,\gerq})=\tYrigtau I$, for a multiset of intervals such that $I_\gerq=\II_\gerq^\ast$. Note that this implies that Lemma \ref{Lemma: saturated at one prime} can be applied for this region. As in Remark \ref{Remark: correspondence on R} (and, in fact, as a byproduct of the proof of Proposition \ref{Proposition: analytic continuaion split}), one shows that 
$\pi_{1,\gerq}^{-1}(\calR_S)\subset \pi_{2,\gerq}^{-1}(\calR_S) \subset \tilde{\gerY}^\gerq_{\rm rig}$ (See \S \ref{Section: overconvergent} for notation).

\begin{lem}\label{Lemma: induction step} Let $S \subset \SS$, and $\gerq \not\in S$. Assume 
\[
\{f_T(S): T \subset S^c\}
\]
is a collection of sections of $\omega^{\underline{\kappa}}$ over $\calR_S$. Assume that for each $T\subset S^c$, the form $f_T(S)$ is an $\HH_{S^c}$-eigenform with system of Hecke eigenvalues $\lambda_{T}$. Assume, further, that for each $T \subset S^c$, and $(\uA,C) \in \calC^\ast$, $f_T(S)(\uA,C)$ is independent of the choice of $C_\gerp$ for $\gerp \in S$. Then, there is a collection of sections of $\omega^{\underline{\kappa}}$ on $\calR_S \cup w_{\gerq}^{-1}(\calR_{S,\gerq})=\calR_{S\cup \{\gerq\}}$,
\[
\{f_T(S\cup\{\gerq\}): T \subset (S\cup\{\gerq\})^c\},
\]
such that, for each $T \subset (S\cup\{\gerq\})^c$, the form $f_T(S\cup\{\gerq\})$ is  an $\HH_{(S\cup\{\gerq\})^c}$-eigenform with system of Hecke eigenvalues $\lambda_T$. Furthermore, if $(\uA,C) \in \calC^\ast$, then,  $f_T(S\cup\{\gerq\})(\uA,C)$ is independent of the choice of $C_\gerp$ for $\gerp \in S\cup\{\gerq\}$. 
Finally, for $T \subset (S\cup\{\gerq\})^c$, if $f_T(S\cup\{\gerq\})$ is classical, then both  $f_T(S)$ and $f_{T\cup \{\gerq\}}(S)$ are classical.
\end{lem}

\begin{proof}
Consider $T \subset (S\cup\{\gerq\})^c$. We want to define $f_T(S\cup \{\gerq\})$. Both $T$ and $T \cup \{\gerq\}$ are subsets of $S^c$, and, hence, $f_T(S)$ and $f_{T\cup \{\gerq\}}(S)$ are defined. Without loss of generality, we assume they are normalized. Therefore, by assumptions, they have the same $q$-expansion coefficients away from $\gerq$.
Applying a $q$-expansion calculation as in Proposition \ref{Proposition: equality of sections}, and using the fact that $\calR_{S,\gera}$ is connected for each choice of $(\gera,\gera^+)$, we can show that on $\pi_{1,\gerq}^{-1}(\calR_S)$ ($\subset \pi_{2,\gerq}^{-1}(\calR_S)$), we have 
\begin{eqnarray}
\pi_{1,\gerq}^*(a_\gerq f_{T\cup \{\gerq\}}(S)-b_{\gerq}f_{T}(S))={\rm pr}^*\pi_{2,\gerq}^*(f_{T\cup \{\gerq\}}(S)-f_T(S)).  \label{section}
\end{eqnarray}
Now, $a_\gerq f_{T\cup \{\gerq\}}(S)-b_{\gerq}f_{T}(S)$ is defined on $\calR_S$, and $w_{\gerq}(f_{T\cup \{\gerq\}}(S)-f_T(S))$ is defined on $w_{\gerq}^{-1}(\calR_{S,\gerq}) \subset w_{\gerq}^{-1}(\calR_S)$. We want to show that these sections agree on
\[
\calR_S \cap w_{\gerq}^{-1}(\calR_{S,\gerq})
\] 
This can be proved using Equation \ref{section}, in combination with Lemma \ref{Lemma: saturated at one prime}, exactly as in Proposition \ref{Proposition: Intersection} (the fact that Lemma \ref{Lemma: saturated at one prime} can be applied was explained before the statement of Lemma \ref{Lemma: induction step}). Gluing these sections, we obtain an extension of $a_\gerq f_{T\cup \{\gerq\}}(S)-b_{\gerq}f_{T}(S)$ to 
\[
\calR_S \cup w_{\gerq}^{-1}(\calR_{S,\gerq})=\calR_{S \cup \{\gerq\}},
\]
which we define to be $f_T(S\cup\{\gerq\})$, and which is clearly an $\HH_{(S\cup\{\gerq\})^c}$-eigenform with system of Hecke eigenvalues $\lambda_T$.

Assume that  $f_T(S\cup\{\gerq\})=a_\gerq f_{T\cup \{\gerq\}}(S)-b_{\gerq}f_{T}(S)$ can be extended to $\Yrig$. The above argument shows that $w_\gerq(f_{T\cup \{\gerq\}}(S)-f_T(S))$, and hence $f_{T\cup \{\gerq\}}(S)-f_T(S)$ can be extended to $\Yrig$.  Since $a_\gerq\neq b_\gerq$, it follows that both $f_T(S)$ and $f_{T\cup \{\gerq\}}(S)$ extend to $\Yrig$ as claimed.


Finally, we need to show that over $\calC^\ast$, $f_T(S\cup\{\gerq\})$ is independent of $C_\gerp$ for $\gerp \in S \cup \{\gerq\}$. This is clearly true for $\gerp \in S$ by assumptions on $f_T(S)$ and $f_{T\cup \{\gerq\}}(S)$. For $\gerp=\gerq$ this follows as in the proof of Corollary \ref{Corollary: average}.
\end{proof}

We now prove Theorem \ref{Theorem: classical general}. We argue by induction starting with $S=\emptyset$, and $\{f_T(\emptyset)\}=\{f_T\}$, and add prime ideals to $S$ until we reach $S=\SS$. Applying Lemma \ref{Lemma: induction step}, as many times as  $|\SS|$, it follows that $f_\emptyset(\SS)$ is defined on $\calR_\SS=\tYrigtau$, and $f_\emptyset(\SS)(\uA,C)$ is independent of choice of $C$ over $\calC^\ast$. Let $h:=\pi_\ast(f_\emptyset(\SS))$. Using these, and arguing as in the proof of Theorem  \ref{Theorem: Main}, it follows that 
\[
(p^g+1)f_\emptyset(\SS)=\pi^\ast(h).
\]
On the other hand, $h$ is a section of $\omega^{\underline{\kappa}}$ defined on $\tXrigtau$. Applying the rigid analytic Koecher principle, as in the proof of Theorem  \ref{Theorem: Main}, it follows that $h$ extends to $\tXrig$. Therefore, $(p^g+1)f_\emptyset(\SS)=\pi^\ast(h)$ extends to $\tYrig$. Applying the last statement in Lemma \ref{Lemma: induction step} successively, it follows that $f_T(S)$ is classical for all $S \subset \SS$ and $T \subset S^c$. In particular, all $f_T=f_T(\emptyset)$ are classical. Finally, $h$ provides the classical Hilbert modular form of level $\Gamma_{00}(N)$ with the same system of prime-to-$p$ Hecke eigenvalues as all the $f_T$'s. 

 \end{document}